\documentclass[twoside]{amsart}
\usepackage{fancyhdr}

\usepackage{latexsym, amsmath,amssymb}

	\usepackage{xcolor}

\newcommand{\cal}{\mathcal}

\renewcommand{\epsilon}{\varepsilon}
\newcommand{\newsection}[1]
{\subsection{#1}\setcounter{theorem}{0} \setcounter{equation}{0}
\par\noindent}

\newtheorem{theorem}{Theorem}

\newtheorem{lemma}[theorem]{Lemma}
\newtheorem{corr}[theorem]{Corollary}

\newtheorem{proposition}[theorem]{Proposition}
\newtheorem{deff}[theorem]{Definition}
\newtheorem{remark}[theorem]{Remark}
\newcommand{\bth}{\begin{theorem}}
\newcommand{\ble}{\begin{lemma}}
\newcommand{\bcor}{\begin{corr}}

\newcommand{\bdeff}{\begin{deff}}

\newcommand{\bprop}{\begin{proposition}}
\newcommand{\ele}{\end{lemma}}
\newcommand{\ecor}{\end{corr}}
\newcommand{\edeff}{\end{deff}}

\newcommand{\eprop}{\end{proposition}}

\newcommand{\Rn}{{\mathbb R}^n}

\newcommand{\la}{\lambda}

\newcommand{\e}{\varepsilon}

\newcommand{\supp}{\text{supp }}
\renewcommand{\Pi}{\varPi}

\renewcommand{\epsilon}{\varepsilon}

\newcommand{\R}{{\mathbb R}}

\newcommand{\1}{{\bf 1}}

\newcommand{\T}{{\mathbb T}}

\newcommand{\lqct}{L^{q_c}(\mathbb T^n)}

\newcommand{\lqc}{L^{q_c}}

\newcommand{\lqcm}{L^{q_c}(M)}

\newcommand{\qn}{Q_\nu}

\newcommand{\laq}{\la^{\frac1{q_c}}}

\newcommand{\off}{\Upsilon^{\text{off}}}

\newcommand{\diag}{\Upsilon^{\text{diag}}}

\newcommand{\far}{\Upsilon^{\text{far}}}

\newcommand{\smo}{\Upsilon^{\text{smooth}}}

\begin{document}

\title{Improved spectral projection estimates}
\thanks{The second author was supported in part by an AMS-Simons travel grant. The third author was supported in part by the NSF (DMS-1665373). }


\author[]{Matthew D. Blair}
\address[M.D.B.]{Department of Mathematics, 
University of New Mexico, Albuquerque, NM 87131}
\email{blair@math.unm.edu} 
\author[]{Xiaoqi Huang}
\address[X.H.]{Department of Mathematics, University of Maryland, College Park. MD 20742}
\email{xhuang49@umd.edu}
\author[]{Christopher D. Sogge}
\address[C.D.S.]{Department of Mathematics,  Johns Hopkins University,
Baltimore, MD 21218}
\email{sogge@jhu.edu}

\begin{abstract}
We obtain new improved spectral projection estimates
on manifolds of non-positive curvature, including 
sharp ones for relatively large spectral windows for
general
tori.  Our results are stronger than those in an
earlier work of the first and third authors \cite{SBLog}, and the 
arguments have been greatly simplified.  We more directly make
use of pointwise estimates that are implicit in the
work of 
B\'erard~\cite{Berard} and avoid the use of weak-type
spaces that were used in the previous works \cite{SBLog} and \cite{Sogge2015improved}.  We also simplify and strengthen the 
bilinear arguments by exploiting the use of 
microlocal $L^2\to L^{q_c}$ Kakeya-Nikodym estimates
and avoiding the  of $L^2\to L^2$ ones as in earlier results.  This allows us to prove new results for manifolds of negative
curvature and some new sharp estimates for tori.  We also have new and improved techniques in two dimensions
for general manifolds of non-positive curvature.
\end{abstract}

\maketitle

\centerline{ \bf In memoriam: {\em Steve Zelditch (1953-2022)}}

\newsection{Introduction}

The purpose of this paper is to improve upon the spectral
projection estimates on compact manifolds of two
of us \cite{SBLog}, while at the same time simplifying
the arguments.  As in the previous work  we shall focus
on obtaining improved estimates for the critical
exponent
\begin{equation}\label{i.1}
q_c=\tfrac{2(n+1)}{n-1}
\end{equation}
on $n$-dimensional compact boundaryless Riemannian
manioflds $(M,g)$ all of whose sectional curvatures
are non-positive.  If we obtain improved $L^{q_c}$-estimates compared to the universal bounds of one of us
\cite{sogge88}, then, by interpolation and dyadic Sobolev estimates, we obtain
improvements for all other exponents $q>2$, including
the difficult range of subcritical exponents
considered in earlier papers by 
Bourgain~\cite{BourgainRestr}, Sogge~\cite{SoggeKaknik},
Sogge and Zelditch~\cite{SoggeZelditchL4} and
Blair and Sogge~\cite{BlairSoggeRefined}, \cite{blair2015refined}, \cite{BlairSoggeToponogov} and \cite{SBLog}.

We are able to obtain our improvements over the earlier
results in \cite{SBLog} by simplifying and strengthening
the two main steps in the proof of the spectral
projection estimates there.  As in \cite{SBLog}
and its predecessor \cite{Sogge2015improved}, we
shall adapt an argument of Bourgain~\cite{BourgainBesicovitch} for Fourier
restriction problems to obtain our results.  This
allows us to split our estimate into two ``heights''.
The greater ``height'' (which is easier to handle) corresponds to the size of
$L^2$-normalized zonal spherical harmonics on round
spheres $S^n$ which saturate the universal 
$L^q$-estimates of one of us \cite{sogge88} for exponents
$q\ge q_c$, while the lesser ``height'' (which is much
harder to handle) corresponds to the size of the
highest weight spherical harmonics on $S^n$ which
saturate the universal bounds for 
$2<q\le q_c$.  See also \cite{sogge86}.

We are able to obtain the needed estimates for
the larger heights in a much more efficient manner
than in the earlier works \cite{SBLog}
and \cite{Sogge2015improved}.  In these works we 
first obtained weak-type $L^{q_c}$-estimates using
the argument from \cite{BourgainBesicovitch} along
with pointwise estimates for smoothed out spectral projection
kernels that were obtained by the techniques
of B\'erard~\cite{Berard}.  We then ``upgraded'' the
weak-type $L^{q_c}$ estimates in the end to 
$L^{q_c}$-estimates using the Lorentz space improvements
of Bak and Seeger~\cite{BakSeegerST} for
``local'' spectral projection operators related to
the ``global'' ones we were considering.  We now
have a much more direct and simple argument based
on B\'erard-type estimates that completely avoids
the use of weak-type estimates and the corresponding
necessary interpolation argument involving the 
local estimates in \cite{BakSeegerST} that
leads to unnecessary losses.  Also, avoiding the use
of weak-type norms in the other main step which
treats the ``lower heights'' simplifies the arguments
there.

As in \cite{SBLog}, the second main step in the
proof of our improved spectral projection
estimates utilizes bilinear techniques of 
Lee~\cite{LeeBilinear}, Tao~\cite{TaoBilinear},
Tao, Vargas and Vega~\cite{TaoVargasVega}, 
Wolff~\cite{WolffCone} and others that were used
to study Fourier restriction problems in 
Euclidean space.  We are able  to make better
use of the bilinear techniques here by another more direct
argument, which, like in the other
main step in the proof, is based on a single
basic estimate and straightforward arguments
using ``local operators''.  The estimate that
we use for this is a ``microlocal Kakeya-Nikodym'' 
estimate going from $L^2(M)$ to $L^{q_c}(M)$ that
involves a microlocalized version of the
smoothed out spectral
projection operators that we ultimately want to control.
These microlocal Kakeya-Nikodym $L^2\to L^{q_c}$ 
estimates were also used by two of us in 
\cite{SBLog}; however, they were also paired with an
additional superfluous step that involved
$L^2\to L^2$ Kakeya-Nikodym estimates (such as in
\cite{BlairSoggeToponogov}) that arose in
all of the earlier works and was an  artifact
of the earliest works in this thread by
Bourgain~\cite{BourgainRestr}  and 
Sogge~\cite{SoggeKaknik}. 
Naturally, removing this unnecessary step involving
$L^2$--Kakeya-Nikodym bounds leads to 
improvements over earlier results. 
We are able to obtain further improvements in 
two-dimensions by tightening  the bilinear 
estimates in this case versus the ones in
this dimension in \cite{SBLog} and
earlier works.

Additionally, in some settings we are able to exploit
specific geometric features to obtain
improved $L^2\to L^{q_c}$ Kakeya-Nikodym estimates
compared to the ones in \cite{SBLog} for general
manifolds of non-positive curvatures.  These
lead to further improvements in certain settings.

When $M$ is a torus, we are able to exploit
favorable commutator properties for operators that
arise in the second step of the proof to obtain
new spectral projection estimates that we show to be
sharp.  These involve relatively large spectral
projection windows compared to those in the earlier
works of Bourgain and Demeter~\cite{BourgainDemeterDecouple}, 
Germain and Myerson~\cite{Germain} and
Hickman~\cite{Hickman} that allowed very
thin windows of width about
the spectral parameter $\la$ that go all the way to the
wavelengths $\la^{-1}$.  To be able to obtain
near sharp estimates for these very narrow spectral
windows these authors used decoupling which resulted
in modest $\la^\e$ losses, while, as we mentioned,
our new estimates are best possible for the
relatively large spectral scales we are able to handle with  bilinear
techniques.  

Also, for manifolds of strictly {\em negative} sectional
curvatures, we are able to exploit the rapid 
dispersive properties of wave kernels to obtain
improvements of the microlocal $L^2\to L^{q_c}$
Kakeya-Nikodym estimates compared to the ones we
can obtain for the general case of manifolds of
non-positive curvature.  By exploiting these, 
we are able, in the case of manifolds of negative curvature to obtain
estimates for spectral windows of width $(\log\la)^{-1}$ about $\la$
which are stronger in two-dimensions than the (sharp) estimates
that we obtain for ${\mathbb T}^2$ and, in 3-dimensions, that match up
with the (sharp) estimates for ${\mathbb T}^3$ for this spectral scale.
 Unfortunately,
unlike in the latter setting, we are unable to show that these
natural estimates are optimal and we discuss a problem
concerning geodesic concentration of log-quasimodes
on  manifolds manifolds of negative curvature that seems to be related to
recent work of Brooks~\cite{BrooksQM}
and Eswarathasan and Nonnenmacher~\cite{ENQM}.
\medskip

Let us now state our main results.  We shall consider smoothed out spectral projection
operators of the form
\begin{equation}\label{i.2}
\rho\bigl(T(\la-P)\bigr), \quad T=T(\la),
\end{equation}
with $P=\sqrt{-\Delta_g}$
where $\rho$ is a real-valued function satisfying
\begin{equation}\label{i.2.2}
\rho\in {\mathcal S}(\R), \quad 
\rho(0)=1, \quad \text{and }
\text{supp }\Hat \rho\subset \{|t|\in(\delta/2,\delta)\},
\end{equation}
where $0<\delta<1$
will be specified later.  For the most part, we shall
take 
\begin{equation}\label{i.3}
T= c_0\log\la,
\end{equation}
with $c_0>0$ ultimately chosen sufficiently small depending on $(M,g)$.  For tori, though, we shall take $T$ to be much larger.

Our main result which improves on estimates
in \cite{SBLog} then is the following

\begin{theorem}\label{thm1}
Assume that $(M,g)$ is an $n$-dimensional compact manifold with non-positive sectional curvatures.  Then, for $T$ as above,
\begin{equation}\label{i.4}
\|\rho(T(\la-P))f\|_{\lqcm} \lesssim \la^{\frac{1}{q_c}} \, \bigl(\log\la\bigr)^{-\sigma_n}\|f\|_{L^2(M)},
\end{equation}
where $\sigma_n=\tfrac{2}{(n+1)q_c}$.
\end{theorem}

This result improves the earlier results of two of the authors in \cite{SBLog} by obtaining a larger power of $\sigma_n$ in all dimensions.
For instance, in \cite{SBLog}, this exponent was $\tfrac{4}{3(n+1)^3}\ll \tfrac{2}{(n+1)q_c}=\tfrac{n-1}{(n+1)^2}$.  The proof of
Theorem~\ref{thm1} is also much simpler than earlier proofs, and, as we shall see leads to further improvements in certain geometries.

Let us note that if $\e(\la)=(\log\la)^{-1}$ then
\eqref{i.4} immediately implies bounds for the
spectral projection operators
\begin{equation}\label{i.5}
\chi_{[\la-\e(\la),\la+\e(\la)]}=
\1_{[\la-\e(\la), \, \la+\e(\la)]}(P),
\end{equation}
with $\1_I$ denoting the indicator function of the
interval $I$.  Indeed since $\rho(0)=1$, by a simple
orthogonality argument, Theorem~\ref{thm1} yields
the following.

\begin{corr}\label{cor1}
If $\chi_{[\la-\e(\la),\la+\e(\la)]}$ is as in
\eqref{i.5} and $(M,g)$ as above has non-positive
curvatures then for $\la\ge1$ and 
$\sigma_n=2/(n+1)q_c$
\begin{equation}\label{i.6}
\bigl\|\chi_{[\la-(\log\la)^{-1} , \, \la+ (\log\la)^{-1}]}f
\bigr\|_{L^{q_c}(M)}\le C\la^{\frac1{q_c}}
(\log\la)^{-\sigma_n}\|f\|_{L^2(M)}.
\end{equation}
\end{corr}

Recall that the universal bounds of one of
us \cite{sogge88} say that
\begin{equation}\label{i.7}
\|\chi_{[\la-1,\, \la+1]}\|_{L^2(M)\to L^q(M)}
\lesssim \la^{\mu(q)}, \quad 2<q\le \infty,
\end{equation}
with
\begin{equation}\label{i.8}
\mu(q)=\max\Bigl( n(\tfrac12-\tfrac1q)-\tfrac12, \,
\tfrac{n-1}2(\tfrac12-\tfrac1q) \Bigr).
\end{equation}
Since $\mu(q_c)=1/q_c$, \eqref{i.6} is a 
log-power improvement of the universal bounds on
manifolds of non-positive curvature.  We need to use
shrinking spectral windows in order to improve
the bounds \eqref{i.7} which are saturated on 
any manifold (see \cite{SFIOII}); however, for
$S^n$ there is no improvement of \eqref{i.6} even
if one uses $(\log\la)^{-1}$--spectral windows as in
\eqref{i.6}.  So the assumption of non-positive
curvature is needed to obtain the improved spectral
projection estimates in \eqref{i.6}.

We also note that, by interpolating with the trivial
$L^2\to L^2$ estimates for the spectral projection
operators, \eqref{i.6} yields
\begin{equation}\label{i.9}
\bigl\|\chi_{[\la-(\log\la)^{-1}, \, \la+ (\log\la)^{-1}]}f
\bigr\|_{L^q(M)}\le \Bigl(\, (\log\la)^{-\frac2{n+1}}
\la\Bigr)^{\mu(q)} \, \|f\|_{L^2(M)}, \quad 2<q\le q_c,
\end{equation}
with 
$\mu(q)=\tfrac{n-1}2(\tfrac12-\tfrac1q)$ as
in \eqref{i.8}.  By using dyadic Sobolev estimates
one can also obtain log-power improvements for
$q>q_c$; however, these would not be as strong
as the $(\log\la)^{-1/2}$ improvements over the universal
bounds for supercritical exponents of 
B\'erard~\cite{Berard} and Hassell and Tacy~\cite{HassellTacy}.

After we prove Theorem~\ref{thm1}, we shall see that
stronger estimates hold if one assumes all of the
sectional curvatures are negative, and that, in this
case the ones obtained for  three-dimensions
have numerology that is related to recent
toral spectral projection bounds of Hickman~\cite{Hickman} and Germain and Myerson~\cite{Germain},
while in two-dimensions are more favorable.  We shall also obtain
in all dimensions $n\ge2$ optimal spectral
projection estimates on 
$n$-dimensional tori for $\e(\la)$-windows
with $\e(\la)$ larger than a fixed negative power of
$\la$ which depends on the dimension.

Next, to motivate the two main steps in our proofs, let us recall the eigenfunctions on $S^n$ that saturate
the universal estimates \eqref{i.6}.  If we wish to
obtain improved bounds, such as those in \eqref{i.4}
or \eqref{i.6}, we need to develop techniques that will
be able to show that the types of extreme eigenfunctions
that exist on standard spheres cannot exist on manifolds
of non-positive curvature.

The eigenfunctions that saturate \eqref{i.7} on $S^n$
with eigenvalue $\la =\sqrt{k(k+n-1)}$ are the
zonal spherical harmonics of degree $k$, $Z_k$.  If
they are $L^2$-normalized and if $B_\pm(c\la^{-1})$
are geodesic balls of radius $c\la^{-1}$ about the
two poles on $S^n$ then
\begin{equation}\label{i.10}
|Z_k(x)|\approx \la^{\frac{n-1}2}, \quad
x\in B_\pm(c\la^{-1}),
\end{equation}
for some sufficiently small uniform constant $c>0$.
(See, e.g., \cite{sogge86}).  As a result, an
easy calculation shows that for some $c_0>0$
 we have the lower bounds
$$\|Z_k\|_{L^q(S^n)}\ge c_0\la^{n(\frac12-\frac1q)-\frac12}, \quad \text{if } \, \,
\la =\sqrt{k(k+n-1)}.
$$
This implies that \eqref{i.7} cannot be improved
when $M=S^n$  for the range of exponents for which
$\mu(q)=n(\tfrac12-\tfrac1q)-\tfrac12$, which
is $q\ge q_c$, where $q_c$ is as in \eqref{i.1}.

The other extreme eigenfunctions for the above
values of $\la$ are the highest weight spherical
harmonics of degree $k$, $Q_k$.   If 
$\gamma \subset S^n$ is the equator, and if ${\mathcal T}_{\la^{-1/2}}(\gamma)$ is a $\la^{-1/2}$-tubular neighborhood
of $\gamma$, then
$$\|Q_k\|_{L^2({\mathcal T}_{\la^{-1/2}}\, (\gamma))}
\approx 1 \quad
\text{if } \, \la =
\sqrt{k(k+n-1)} \quad \text{and } \, \,
\|Q_k\|_{L^2(S^n)}=1.$$
(See, e.g. \cite{sogge86}.)  Therefore, as we shall
show in \S\ref{g}, it is a simple exercise
using H\"older's inequality to verify that for
some uniform $c_0>0$ 
\begin{equation}\label{i.11}
\|Q_k\|_{L^q(S^n)}\ge c_0 \, \la^{\frac{n-1}2(\frac12
-\frac1q)}, \quad 2<q\le \infty.
\end{equation}
As a result, \eqref{i.7} cannot be improved for
$2<q\le q_c$ since $\mu(q)=\tfrac{n-1}2(\tfrac12
-\tfrac1q)$ for such exponents.  To further motivate
the splitting that we shall use in order to prove
Theorem~\ref{thm1}, we recall (see \cite{sogge86})
that the lower bounds in \eqref{i.11} are actually
upper bounds, and, in particular
\begin{equation}\label{i.12}
\|Q_k\|_{L^\infty(S^n)}\approx \la^{\frac{n-1}4},
\quad \text{if } \, \, \la=\sqrt{k(k+n-1)}.
\end{equation}

We note that both the zonal spherical harmonics, $Z_k$,
and the highest weight spherical harmonics, $Q_k$, saturate the 
universal bounds \eqref{i.7} when $q$ equals the critical exponent $q_c$ in \eqref{i.1}.  So estimates
involving this exponent are sensitive to both ``point
concentration'' and ``geodesic concentration''.  The
two main steps in the proof of Theorem~\ref{thm1} will
rule out these types of extreme concentration under
our curvature assumptions, and, based on the arguments
from \S\ref{g} and \eqref{i.12}, the ``height'' at which
the transition between these types of concentration
can occur should be $\la^{\frac{n-1}4}$.

\noindent{\bf Acknowledgements:}  We are deeply indebted 
to Steve Zelditch for the encouragement and generous
advice that he has given us over the years on this project and its predecessors.  In particular, his 
insistence that we should use microlocal estimates
rather than the original $L^2$ spatial Kakeya-Nikodym
estimates originally used by one of \cite{SoggeKaknik} was crucial
for the progress that we have made on these problems
over the years.  Moreover, the use
of microlocal estimates like \eqref{7} here has
strengthened our bounds and greatly simplified
the arguments.

We also are very grateful for the helpful suggestions of the referees which significantly improved
the exposition.

\newsection{The height splitting}

As in some of our earlier works, we shall employ
bilinear techniques that require us to compose
the ``global operators''
\begin{equation}\label{i.13}
\rho_\la =\rho(T(\la-P))
\end{equation}
in Theorem~\ref{thm1} with the ``local operators''
\begin{equation}\label{i.14}
\sigma_\la = \rho(\la-P).
\end{equation}

Also, as in earlier works, it will be convenient to localize a bit more.  To this end, let us write
\begin{equation}\label{i.b1}
I=\sum_{j=1}^N B_j(x,D),
\end{equation}
where each $B_j\in S^0_{1,0}(M)$ is a standard zero order pseudo-differential with
symbol supported in a small conic neighborhood of some $(x_j,\xi_j)\in S^*M$.
The size of the support will be described in \S4; however, we note now that
these operators will not depend on our spectral parameter $\la \gg 1$.  Also note that if
$\beta\in C^\infty_0((1/2,2))$ is a Littlewood-Paley bump function which equals
one near $t=1$ then the dyadic operators
\begin{equation}\label{i.b2}
B=B_{j,\la}=B_j\circ \beta(P/\la)
\end{equation}
are uniformly bounded on $L^p$, i.e.,
\begin{equation}\label{i.b3}
\|B\|_{L^p(M)\to L^p(M)}=O(1) \quad \text{for } \, \, 1\le p\le \infty.
\end{equation}
See e.g., \cite[Chapter II]{TaylorPDO}
for definitions of the above symbol classes, and the bounds in \eqref{i.b3} follow from
the results in \cite[Theorem 3.1.6]{SFIOII} if $1<p<\infty$ and the proof of this result
if $p=1$ or $p=\infty$.

We then shall further localize $\sigma_\la$ by setting
\begin{equation}\label{i.b4}
\tilde \sigma_\la = B\circ \sigma_\la,
\end{equation}
where $B$ is one of the $N$ operators in \eqref{i.b2}.
We also define 
the ``semi-global'' operators
\begin{equation}\label{i.15}
\tilde \rho_\la = \tilde \sigma_\la \circ \rho_\la.
\end{equation}

We note that the $\sigma_\la$ can be thought of as
``smoothed out'' versions of the operators in \eqref{i.7}.  They enjoy the same 
operator norms and the two sets of estimates are equivalent.
Similarly, it is an easy exercise using orthogonality
to see that we have the uniform bounds
\begin{multline}\label{tilde}
\|(I-\sigma_\la)\circ \rho_\la\|_{L^2(M)\to L^{q_c}(M)}
\le CT^{-1}\la^{\frac1{q_c}}, \, \, \text{if } \, 
T\ge 1, 
\\
 \text{and } \, \, \, \|\sigma_\la-\beta(P/\la)\circ \sigma_\la\|_{L^2(M)\to L^{q_c}(M)}=O(\la^{-N}), \, \, \, \forall \, N.
\end{multline}
Therefore by \eqref{i.b1} and \eqref{tilde}, in order to prove \eqref{i.4}, it suffices to
show that
\begin{equation}\label{0.1}
\|\tilde \rho_\la f\|_{L^{q_c}(M)}
\lesssim \la^{\frac1{q_c}}(\log\la)^{-\sigma_n}
\|f\|_{L^2(M)},
\end{equation}
if $T$ is as in \eqref{i.3}.

Due to the bilinear arguments that we shall employ, we need to make the height decomposition using
the semi-global operators
$\tilde \rho_\la$.  We shall always, as we may, assume
that the function $f$ in \eqref{0.1} is $L^2$-normalized, 
i.e.,
\begin{equation}\label{i.18}
\|f\|_{L^2(M)}=1.
\end{equation}
We shall then split our task \eqref{0.1} into estimating
the $L^{q_c}$-norm of $\tilde \rho_\la f$ over the two
regions,
\begin{equation}\label{i.19}
A_+=\{x\in M: \, |\tilde \rho_\la f(x)|\ge 
\la^{\frac{n-1}4+\frac18} \, \}
\end{equation}
and
\begin{equation}\label{i.20}
A_-=\{x\in M: \, |\tilde \rho_\la f(x)|< 
\la^{\frac{n-1}4+\frac18} \, \},
\end{equation}
which basically correspond to the height \eqref{i.12}
of the highest weight spherical harmonics.  There
is nothing special about the exponent $1/8$ in \eqref{i.19} and \eqref{i.20}.
Just as in \cite{SBLog}, it could be replaced by any
sufficiently small $\delta>0$ in what follows.

\newsection{Large height estimates}

Let us describe here how to estimate the $L^{q_c}(A_+)$
norm of $\tilde \rho_\la f$.
If $a\in C^\infty_0((-1,1))$ equals one on $(-1/2,1/2)$,
then we can do this using the ``global estimate''
\begin{multline}\label{2}
G_\la(x,y)=\frac1{2\pi} \int_{-\infty}^\infty \bigl(1-a(t)\bigr) \, T^{-1} \Hat\Psi(t/T) e^{it\la} \, \bigl(e^{-itP})(x,y)\, dt
\\
=O\bigl(\la^{\frac{n-1}2}\, \exp(C_0T)\Bigr), \, \, \, \Psi=\rho^2 \, \, \, 1\le T\lesssim \log\la.
\end{multline}

This estimate is valid when $(M,g)$ has nonpositive sectional curvatures 
(see e.g., \cite{Berard}, \cite{HassellTacy}, \cite{Sogge2015improved},
\cite{SoggeHangzhou}).  One proves \eqref{2}
by standard arguments after lifting the calculation
up to the universal cover.  Since $\Hat \Psi$ is compactly supported, the number of terms in the sum that 
arises grows exponentially in $T$ if $(M,g)$ has negative
curvature, and this accounts for the $\exp(C_0T)$ factor
in \eqref{2}. (See e.g., \cite[\S 3.6]{SoggeHangzhou} and the related arguments in \S\ref{neg} below). Later in \S\ref{torus} we shall be able 
to obtain more favorable estimates for tori than those
in Theorem~\ref{thm1} in part by using the fact that
in this setting, the number of terms grows polynomially.  Due
to this, for tori we shall use a slightly different splitting
into the two heights compared to \eqref{i.19}
and \eqref{i.20}.

Let us now see that if \eqref{2} is valid then we have the following.

\begin{proposition}\label{prop2}  Let $f$ and
$A_+$ be as in \eqref{i.18} and \eqref{i.19} and $B$ be as in \eqref{i.b2}.
Then for $\la \gg 1$, if $T=c_0\log\la$ with $c_0>0$ sufficiently small,
\begin{equation}\label{3}
\|\tilde \rho_\la f\|_{\lqc(A_+)}\le C\laq \, (\log\la)^{-\frac{1}{2}} \, \|f\|_{L^2(M)}.
\end{equation}
\end{proposition}

\begin{proof}
We first note that, by \eqref{i.b3} and \eqref{tilde}, we have 
$$\| \tilde\rho_\la f\|_{L^{q_c}(A_+)}\le \|B\rho_\la f\|_{L^{q_c}(A_+)}+C\la^{\frac1{q_c}}/\log \la,$$
since we are assuming that $f$ is $L^2$-normalized.  Thus, we would have \eqref{3} if we could show that
\begin{equation}\label{3'}
\|B\rho_\la f\|_{L^{q_c}(A_+)}\le C\la^{\frac1{q_c}}(\log\la)^{-1/2}+\tfrac12 \|\tilde \rho_\la f\|_{L^{q_c}(A_+)}.
\end{equation}

To prove this we shall adapt an argument of Bourgain~\cite{BourgainBesicovitch} and simplify related ones in
\cite{Sogge2015improved} and \cite{SBLog}.  As mentioned before, we shall avoid using arguments that utilize the Lorentz space estimates of Bak and Seeger~\cite{BakSeegerST}, unlike in \cite{Sogge2015improved} and \cite{SBLog}.

To this end, choose $g$ such that 
$$\|g\|_{L^{q'_c}(A_+)}=1, \quad \text{and } \, \, \, 
\| B\rho_\la f\|_{L^{q_c}(A_+)} = \int  B\rho_\la f \, \cdot
\bigl(\1_{A_+} \cdot g\bigr) \, dx.$$
Then since $ \rho^*_\la= \rho_\la$ and $\Psi(T(\la-P))= \rho_\la \circ  \rho_\la$ for $\Psi$ as in \eqref{2},  by the Schwarz inequality, since we are assuming that $\|f\|_{2}=1$ we have
\begin{align}\label{4}
\|  B\rho_\la f\|^2_{L^{q_c}(A_+)}&=\Bigl(\, \int f \cdot \bigl(\rho_\la B^*\bigr)(\1_{A_+}\cdot g)(x) \, dx \, \Bigr)^2
\\
&\le \int \bigl| \, \rho_\la B^*\bigl(\1_{A_+}\cdot g\bigr)(x)\, \bigr|^2 \, dx \notag
\\
&=\int \bigl(B\circ\Psi(T(\la-P))\circ B^*\bigr)(\1_{A_+}\cdot g)(x) \, \cdot \, \overline{\1_{A_+}(x) \, g(x)} \, dx \notag
\\
&=\int \bigl(B\circ L_\la\circ B^*\bigr)(\1_{A_+}\cdot g)(x) \cdot \overline{\1_{A_+}(x) \, g(x)} \, dx \notag
\\
&\qquad \qquad + \int \bigl(B\circ G_\la \circ B^*\bigr)(\1_{A_+}\cdot g)(x) \cdot \overline{\1_{A_+}(x) \, g(x)} \, dx \notag
\\
&= I +II, \notag
\end{align}
with here, and in what follows $\|h\|_p=\|h\|_{L^p(M)}$.
Here, $G_\la$ is the operator whose kernel is as in \eqref{2}, while $L_\la$ is the ``local'' operator
$$L_\la =(2\pi T)^{-1} \int a(t) \Hat \Psi(t/T) e^{it\la} e^{-itP}\, dt.$$
Thus, $L_\la h=T^{-1}\sum_j m(\la;\la_j) E_jh$, where $E_j$ denotes the projection onto the $j$th eigenspace of $P$ and
the spectral multiplier $m(\la;\la_j)= (\check{a}*\Psi_T)(\la-\la_j)$, $\Psi_T(\tau)=T\Psi(T\tau)$,
 satisfies
$$m(\la;\la_j)=O((1+|\la-\la_j|)^{-N}), \quad T\ge1,$$
for any $N$.  Consequently, by the universal bounds in \cite{sogge88}, we have
$$\|L_\la\|_{L^{q_c'}(M) \to L^{q_c}(M)} \lesssim T^{-1}\la^{\frac{2}{q_c}}.$$
Since $T=c_0\log\la$, if we use H\"older's inequality and \eqref{i.b3} we find that
\begin{align}\label{5}
|I| &\le \bigl\|BL_\la B^*\bigr(\1_{A_+}\cdot g\bigr)\bigr\|_{q_c}\cdot\| \1_{A_+} \cdot g\|_{q_c'}
\\
&\lesssim \bigl\|L_\la B^*\bigr(\1_{A_+}\cdot g\bigr)\bigr\|_{q_c}\cdot\| \1_{A_+} \cdot g\|_{q_c'} \notag
\\
&\lesssim \la^{\frac{2}{q_c}}(\log\la)^{-1}  \bigl\| B^*\bigr(\1_{A_+}\cdot g\bigr)\bigr\|_{q_c'}\cdot\| \1_{A_+} \cdot g\|_{q_c'} \notag
\\
&\lesssim \la^{\frac{2}{q_c}}(\log\la)^{-1}\|g\|^2_{L^{q_c'}(A_+)} \notag
\\
&=\la^{\frac{2}{q_c}}(\log\la)^{-1}. \notag
\end{align}

To estimate $II$, we choose $c_0>0$ small enough
so that if $C_0$ is the constant in \eqref{2} then
$$\exp(C_0T)\le \la^{\frac{1}{8}} \quad \text{if } \, \, T=c_0\log\la \quad \text{and } \, \, \la \gg 1.$$
Consequently, \eqref{2} yields
$$\|G_\la\|_{L^1(M)\to L^\infty(M)}\le \la^{\frac{n-1}2+\frac18}.$$
As a result, since the dyadic operators $B$ are uniformly bounded on $L^1$ and $L^\infty$, we can
repeat the argument that we used to estimate $I$ to see that
\begin{equation*}
|II|\le C\la^{\frac{n-1}2}\la^{\frac18} \, \bigl\| \1_{A_+} \cdot g\bigr\|_1^2 \le C\la^{\frac{n-1}2}\la^{\frac18} \,
\|g\|^2_{L^{q_c'}(A_+)} \, \|\1_{A_+}\|_{q_c}^2
=C\la^{\frac{n-1}2}\la^{\frac18} \, \|\1_{A_+}\|_{q_c}^2.
\end{equation*}
If we recall the definition \eqref{i.19} of $A_+$, we can estimate the last factor:
$$\|\1_{A_+}\|_{q_c}^2 \le \bigl(\la^{\frac{n-1}4 +\frac18}\bigr)^{-2}\|\tilde \rho_\la f\|^2_{L^{q_c}(A_+)}.$$
Therefore,
$$|II|\lesssim \la^{-1/8}\| \tilde \rho_\la f\|^2_{L^{q_c}(A_+)} \le \bigl(\tfrac12\| \tilde \rho_\la f\|_{L^{q_c}(A_+)}\bigr)^2,
$$
assuming, as we may, that $\la$ is large enough.

If we combine this bound with the earlier one, \eqref{5}, for $I$ we conclude that
\eqref{3'} is valid
which of course yields \eqref{3} and completes the proof of Proposition~\ref{prop2}.
\end{proof}

\newsection{Controlling small heights using microlocal Kakeya-Nikodym estimates}

Note that, under our curvature assumption, Proposition~\ref{prop2} rules out the existence of eigenfunctions
like the zonal spherical harmonics $Z_k$ on $S^n$ that
maximally concentrate near points.  On the other hand,
given \eqref{i.12} it does not rule out the existence
of eigenfunctions like the highest weight spherical harmonics since the $Q_k$ vanish on $A_+$.
To complete the proof of Theorem~\ref{thm1} we also need the following result for the complementary set, $A_-$,
which does rule out the existence of these types of
eigenfunctions with maximal concentration near geodesics.

\begin{proposition}\label{prop3}  Let $M$ have nonpositive curvature.  Then
\begin{equation}\label{6}
\|\tilde \rho_\la f\|_{L^{q_c}(A_-)} \le C\la^{\frac1{q_c}} \, (\log\la)^{-\sigma_n} \, \|f\|_{L^2(M)},
\end{equation}
where $\sigma_n$ is as in Theorem~\ref{thm1}.
\end{proposition}

Let us collect the tools from \cite{SBLog} that we shall need to prove \eqref{6}.   We first
recall that the symbol $B(x,\xi)$ of $B$ in \eqref{i.b2} 
 is supported in a conic
neighborhood of some $(x_0,\xi_0)\in S^*M$.  We may
assume that its symbol has small enough support so that
we may work in a coordinate chart $\Omega$ so that
$x_0=0$, $\xi_0=(0,0,\dots,0,1)$
and $g_{jk}(0)=\delta^j_k$ in the local coordinates.
So we shall assume that the $x$-support of $B(x,\xi)$
is contained in $\Omega$.  We also may assume that
$B(x,\xi)$ vanishes when $\xi$ is outside of a small
neighborhood of $(0,\dots,0,1)$.  These reductions and those that
follow will contribute to the number of summands in \eqref{i.b1}; however, it will be clear that the
$N$ there will be independent of $\la\gg 1$.

The bilinear arguments we shall need to use will involve
microlocal cutoffs corresponding to angular sectors of 
aperture $\approx \la^{-1/8}$.  To construct them
write $x=(x',x_n)$ so that $x'=(x_1,\dots,x_{n-1})$
denotes coordinates on the hypersurface where
$x_n=0$ in our coordinates.  We shall want these
cutoffs to commute with the unit-speed geodesic flow
$\chi_t: \, S^*M\to S^*M$ on the support of $B(x,\xi)$
if $|t|\le \delta$, where $\delta$ is as in 
\eqref{i.2.2}.  Recall that $\chi_t$ is the flow
generated by the Hamilton vector field associated with
the principal symbol
$$p(x,\xi)=\bigl(\, \sum_{j,k}g^{jk}(x)\xi_j\xi_k\, \bigr)^{1/2}$$
of $P=\sqrt{-\Delta_g}$.  Here $g^{jk}(x)=\bigl(g_{jk}(x)\bigr)^{-1}$ is the cometric and so
$S^*M=\{(x,\xi): \, p(x,\xi)=1\}$.

If $|t|\le 2\delta$ with $\delta>0$ small enough then the 
map
\begin{equation}\label{map}
(t,x',\eta)\to \chi_t(x',0,\eta)\in S^*M, \quad
(x',0,\eta)\in S^*M
\end{equation}
is a diffeomorphism from a neighborhood of $x'=0$, $t=0$
and $\eta=(0,\dots,0,1)$ to a neighborhood of
$(x_0,\xi_0)=(0, (0,\dots,0,1))\in S^*M$.  Indeed, since we are 
assuming that $g_{jk}(0)=\delta^j_k$, the Jacobian of the
map is the identity at $(x_0,\xi_0)$.  After possibly
shrinking the support, we may assume that these 
properties are valid on a neighborhood of $\text{supp } B(x,\xi)$. 

Write the inverse of \eqref{map} as
\begin{equation}\label{map'}
S^*M\ni (x,\omega)\to 
\bigl(\tau(x,\omega), \Phi(x,\omega), \Theta(x,\omega)\bigr) \in  (-\delta,\delta) \times \{ y'\in 
{\mathbb R}^{n-1} \} \times S^*_{(\Phi(x,\omega),0)}M.
\end{equation}
Thus, the unit speed geodesic passing through
$(x,\omega)\in S^*\Omega$ arrives at the plane in our local coordinates where
$y_n=0$ at $\Phi(x,\omega)$, has covector
$\Theta(x,\omega)\in S^*_{(\Phi(x,\omega),0)}M$ there, and $\tau(x,\omega)
=d_g(x,(\Phi(x,\omega),0))$ is the geodesic
distance between $x$ and the point $(\Phi(x,\omega),0)$ on this plane.

We can now define the microlocal cutoffs that we shall
use.  First, we let $\nu$ 
range over a 
$\la^{-\frac18}$--separated
set in $S^{n-1}$ lying in our conic neighborhood
of $\xi_0=(0,\dots,0,1)$.  We then choose a
partition of unity $\sum\beta_\nu(\xi)=1$ on this set
which includes a neighborhood of the $\xi$-support of $B(x,\xi)$
so that each $\beta_\nu$ is supported in a $2\la^{-\frac{1}{8}}$ cap about $\nu\in S^{n-1}$, and so that if 
$\beta_\nu(\xi)$ is the homogeneous of degree zero
extension of $\beta_\nu$ to $\Rn\backslash 0$ we have
$$|D^\alpha \beta_\nu(\xi)|\le C_\alpha
\la^{|\alpha|/8} \quad \text{if } \, \,
|\xi|=1.$$
Finally, if $\psi \in C^\infty_0(\Omega)$ equals
one in a neighborhood of the $x$-support of $B(x,\xi)$,
and if $\tilde \beta\in C^\infty_0((0,\infty))$ equals one in a neighborhood of the support of the Littlewood-Paley bump function
in \eqref{i.b2} we define
\begin{equation}\label{qnusymbol}
q_\nu(x,\xi)=\psi(x) \, \beta_\nu\big(\Theta(x, \xi/p(x,\xi)\bigr) \, \tilde\beta\bigl(p(x,\xi)/\la\bigr).
\end{equation}
It then follows that the pseudo-differential operators
$Q_\nu$ with these symbols belong to a bounded subset
of $S^0_{7/8,1/8}(M)$.  We have constructed these operators so that
\begin{equation}\label{egorov99}
q_\nu(x,\xi)=q_\nu\bigl(\chi_t(x,\xi)\bigr) \quad
\text{on \, supp }B(x,\xi) \, \, \text{if } \, \,
|t|\le 2\delta.
\end{equation}
Since
$$\sigma_\la =\frac1{2\pi} \int e^{it\la} 
e^{-itP} \, \Hat \rho(t)\, dt$$
and
$\Hat \rho(t)=0$, $|t|\ge \delta$, as noted in
\cite[(3.24)]{SBLog} one can use Egorov's theorem and the universal estimates in \cite{sogge88}
to see that
\begin{equation}\label{commute} 
\bigl\| \tilde \sigma_\la Q_\nu -Q_\nu 
\tilde  \sigma_\la 
\bigr\|_{L^2(M)\to L^{q_c}(M)} \lesssim
\la^{-\frac14}.
\end{equation}
In the proof of Theorem~\ref{thm1} we only need weaker $O(\la^{\frac1{q_c}-\frac14})$ bounds, which, by a simple exercise,
follow from \eqref{egorov99}, Sobolev estimates and the bounds in \cite{sogge88}.
If $\delta>0$ in \eqref{i.2.2} is small enough
we also have 
\begin{equation}\label{add}
\tilde \sigma_\la -\sum_\nu \tilde \sigma_\la Q_\nu
=R_\la \quad
\text{where } \, \|R_\la\|_{L^p\to L^p}=O(\la^{-N})
\quad \forall \, N \, \, \text{if } \, \,
1\le p \le \infty.
\end{equation}
Finally, the support properties of the symbols imply that
we have the uniform bounds
\begin{equation}\label{alor}
\sum_\nu \|Q_\nu h\|^2_{L^2(M)}\le C\|h\|^2_{L^2(M)}.
\end{equation}

Next, the  global estimates that we shall use to prove the small heights bounds in Proposition~\ref{prop3}
 are the following
``global'' {\em ``$L^{q_c}$-microlocal Kakeya-Nikodym estimates":}
\begin{equation}\label{7}
\sup_\nu \|Q_\nu \rho_\la\|_{L^2(M)\to L^{q_c}(M)}\lesssim \la^{\frac1{q_c}} \cdot \bigl(\log\la \bigr)^{-\frac1{q_c}}.
\end{equation}
This, like \eqref{2}, is valid when $(M,g)$ has non-positive sectional curvatures, although in the next section we shall see that stronger bounds hold if the sectional curvatures all assumed to be negative.
The inequality \eqref{7} is a slight improvement
over the corresponding estimate (2.18) in \cite{SBLog}.

One obtains \eqref{7} via a simple interpolation argument and the following pointwise microlocalized estimates
\begin{multline}\label{14}
K_{\la,\mu}(x,y)=
T^{-1}\int \Hat \Psi(t/T) \beta(t/\mu) \, e^{it\la} \Bigl(Q_\nu \circ \cos(tP)\circ Q_\nu^*\Bigr)(x,y)\, dt 
\\
=O\bigl(T^{-1} \, \mu^{1-\frac{n-1}2} \la^{\frac{n-1}2}\bigr), \, \, \Psi = (\rho)^2, \, \, 
\quad 1\le \mu\le T=4c_0\log\la,
\end{multline}
if $\beta\in C^\infty_0((1/2,2))$ is a Littlewood-Paley bump function
satisfying $1=\sum_{-\infty}^\infty \beta(t/2^j)$ for $t>0$,
and $c_0>0$ is sufficiently small.  To prove this
one uses the Hadamard parametrix after lifting
the calculation up to the universal cover.
The argument is almost identical to the proof of
(3.8) in \cite{BlairSoggeToponogov} or
the proof of Theorem 2.1 in \cite{SBLog}.  The 
pointwise bounds in \eqref{14} are more
favorable than those in \eqref{2} since the result of the 
microlocalization in \eqref{14}  is that
number of significant terms in the calculation only grows linearly in $T$, unlike in the
proof of \eqref{2} where there is exponential growth.  
For more details, see the remark at the 
end of \S\ref{neg}.

By \eqref{14}, the integral operator with kernel
$K_{\la,\mu}$ maps $L^1(M)\to L^\infty(M)$ with norm
$O\bigl(T^{-1} \, \mu^{1-\frac{n-1}2} \la^{\frac{n-1}2}\bigr)$.  By orthogonality it also maps $L^2\to L^2$ with norm
$O(T^{-1}\mu)$.  So by interpolation, it maps $L^{q_c'}\to L^{q_c}$ with norm $O(T^{-1}\mu^{\frac2{n+1}} \la^{\frac{n-1}{n+1}} )$.  
By taking $\mu=2^j\lesssim T$ and adding up these
dyadic estimates we find that if $a$ is as in \eqref{2} we have that the operator with kernel
$$T^{-1}\int \Hat \Psi(t/T) (1-a(t)) \,  e^{it\la}\Bigl(Q_\nu \circ \cos(tP)\circ Q_\nu^*\Bigr)(x,y)\, dt $$
maps $L^{q_c'}\to L^{q_c}$ with norm $O((\la/T)^{\frac2{q_c}})$.  Finally, since the local estimates from \cite{sogge88} show that when we
replace $(1-a(t))$ above by $a(t)$ the resulting integral operator maps $L^{q_c'}\to L^{q_c}$ with a better norm norm $O(T^{-1}\la^{\frac2{q_c}})$, we must have
\begin{equation}\label{99}\sup_\nu \|Q_\nu \Psi(T(\la-P)) Q_\nu^*\|_{L^{q_c'}(M)\to L^{q_c}(M)}\lesssim \la^{\frac2{q_c}} \cdot \bigl(\log\la \bigr)^{-\frac{2}{q_c}},
\end{equation}
which implies \eqref{7} by a simple ``$TT^*$'' argument.

\medskip
\noindent{\bf Bilinear decomposition}

We also require the (local) bilinear harmonic arguments from \cite{SBLog} that take advantage of the 
fact that the norm in \eqref{6} is taken over $A_-$.  These techniques go back to the bilinear Fourier restriction arguments from
 Tao~\cite{TaoBilinear}, Tao, Vargas and Vega~\cite{TaoVargasVega}, Lee~\cite{LeeBilinear}, Wolff~\cite{WolffCone}  and others. 

We shall write using \eqref{i.18} and  \eqref{add}
\begin{equation}\label{add'}
\bigl(\tilde \rho_\la f\bigr)^2
=\sum_{\nu, \tilde \nu} 
(\tilde \sigma_\la Q_\nu h) \cdot
(\tilde \sigma_\la Q_{\tilde \nu}h)+O(\la^{-N}), \quad
h=\rho_\la f.
\end{equation}
We shall organize the pairs of directions
$(\nu,\tilde\nu)\in S^{n-1}\times S^{n-1}$ just as
in \cite{TaoVargasVega}.  So, let us write
$\nu=(\nu',\nu_n)$ where $\nu'\in {\mathbb R}^{n-1}$
is near the origin since $\nu$ is close
to $(0,\dots,0,1)$.

To this end, consider the collection of 
dyadic cubes $\{\tau^j_{\mu'}\}$ in ${\mathbb R}^{n-1}$
of side length $2^j$, where $\tau^j_{\mu'}$ denotes
the translation of $[0,2^j)^{n-1}$ by 
$\mu'\in 2^j {\mathbb Z}^{n-1}$.  Two such cubes
are said to be ``close'' if they are not adjacent
but have adjacent parents of sidelength $2^{j+1}$.
If this is the case, we write $\tau^j_{\mu'}
\sim \tau^j_{\tilde \mu'}$.  Close cubes are separated
by a distance which is comparable to $2^j$.  As noted
in \cite{TaoVargasVega}, any distinct $\nu', \tilde
\nu'\in {\mathbb R}^{n-1}$ lie in a unique pair of
close cubes.  Thus, there is a unique triple 
$j,\mu'\tilde \mu'$ such that
$(\nu', \tilde \nu')\in \tau^j_{\mu'}\times
\tau^j_{\tilde \mu'}$ and $\tau^j_{\mu'}\sim
\tau^j_{\tilde \mu'}$.  Since
$\nu'$ is close to the origin for our $\nu\in S^{n-1}$,
we need only to consider integers $j\le 0$.

If we then let $J \ll 0$ be the integer satisfying
$2^{J-1}<8\la^{-1/8}\le 2^J$, it follows that the
sum in \eqref{add'} can be organized as
\begin{equation}\label{org}
\left( \sum_{J+1\le j\le 0} \,  \, 
\sum_{(\nu',\tilde \nu')\in
\{\tau^j_{\mu'}\times \tau^j_{\tilde \mu'}: \, 
\tau^j_{\mu'}\sim \tau^j_{\tilde \mu'}\}}
+\sum_{(\nu',\tilde \nu')\in \Xi_J}
\right)
\bigl(\tilde \sigma_\la Q_\nu h\bigr) \cdot
\bigl(\tilde \sigma_\la Q_{\tilde \nu}h\bigr),
\end{equation}
where $\Xi_J$ denotes the remaining pairs not included
in the first sum.  These include diagonal pairs where
$\nu'=\tilde \nu'$, and all $(\nu',\tilde \nu')\in
\Xi_J$ satisfy $|\nu'-\tilde \nu'|\lesssim \la^{-1/8}$.
Thus, for each fixed $\nu'$, we have
$\#\{\tilde \nu': \, (\nu', \tilde \nu')\in
\Xi_J\}=O(1)$.

As in \cite{SBLog}, write
\begin{multline}\label{diag}
\diag(h)=\sum_{(\nu',\tilde \nu')\in
\Xi_J}
\bigl(\tilde \sigma_\la Q_\nu h\bigr) \cdot
\bigl(\tilde \sigma_\la Q_{\tilde \nu}h\bigr), 
\\
 \text{if } \, \, 
\la^{-1/8}\in (2^{J-4}, 2^{J-3}], \, \, \text{and } \quad h=\rho_\la f.
\end{multline}
We shall estimate the $L^{q_c/2}$-norm of this ``diagonal'' term  using results from \cite{SBLog}.

Next, if we consider the first sum in the left of  \eqref{org} over $(\nu',\tilde\nu')\notin \Xi_J$ along with the error term in \eqref{add'},
\begin{equation}\label{ndsum}
\far(h)=
\sum_{J+1\le j\le 0}
\sum_{(\nu',\tilde \nu')\in
\{\tau^j_{\mu'}\times \tau^j_{\tilde \mu'}: \, 
\tau^j_{\mu'}\sim \tau^j_{\tilde \mu'}\}}
\bigl(\tilde \sigma_\la Q_\nu h\bigr) \cdot
\bigl(\tilde \sigma_\la Q_{\tilde \nu}h\bigr)+O(\la^{-N}),
\end{equation}
we have
the favorable bilinear estimates
\begin{equation}\label{far}
\int |\far(h)|^{q/2} \, dx
 \le C
\la \, 
\bigl(\la^{\frac78}\bigr)^{\frac{n-1}2(q-q_c)}
\|h\|_{L^2(M)}^q, \, \, 
\text{if } \, \, q\in (\tfrac{2(n+2)}n, q_c).
\end{equation}
For, in the notation of \cite{SBLog}, 
$\far(h)$ here equals
$\off(h)+\smo(h)$ together with
the error term in \eqref{add'}, and so 
\eqref{far} follows from 
summing over $j\in [J+1,0]$ in inequality
(4.7) in \cite{SBLog} and the fact that, like the
error term in \eqref{add'},
$\| \smo(h)\|_{q_c/2}\lesssim
\la^{-N}\|h\|^2_2$ for all $N$, due to 
(4.5) in \cite{SBLog}.  Based on this, one obtains
\eqref{far}.  We note that the estimates
in \cite{SBLog} for the non-trivial part,
$\off$ of $\far$ are based
on the bilinear oscillatory integral estimates
of Lee~\cite{LeeBilinear}.  The terms $\smo$ in
\cite{SBLog} were trivial microlocal error terms to allow us to use 
those estimates via parabolic scaling.  We have opted to lump together these
two terms into \eqref{ndsum} to simplify the notation and argument a tiny bit in what follows.

The significance of \eqref{far} is that the
second power of $\la$ in the right is negative
as $q<q_c$ there.  On the other hand, this makes it somewhat
awkward to take advantage of this power-improvement
to prove the $L^{q_c}$ bound in our remaining
inequality \eqref{6}.  However, as in \cite{SBLog}, we shall
be able to use \eqref{far} since  the norm in
the left hand side of \eqref{6} is taken
over $A_-$.
We also note that, we have rewritten the left side of \eqref{add'} as follows
\begin{equation}\label{add''}
(\tilde  \rho_\la f)^2 =(\tilde \sigma_\la h)^2= \diag(h) + \far(h), \quad h= \rho_\la f.
\end{equation}

 
Next, let us use the fact that we can use \eqref{far} and \eqref{add''} to prove a slightly
stronger version of estimates in \cite{SBLog} that will allow us to prove \eqref{6}.
Specifically,  as we shall show in the next subsection, we can use \eqref{far} to obtain the following
\begin{equation}\label{8}
\|\tilde \sigma_\la h\|_{L^{q_c}(A_-)}
\le C\la^{\frac1{q_c}} \la^{-\delta_n} \|h\|_{L^2(M)}
+C\Bigl(\sum_\nu \|\tilde \sigma_\la \qn h\|_{L^{q_c}(M)}^{q_c}\Bigr)^{1/q_c}, \, \, 
h=\rho_\la f,
\end{equation}
for certain $\delta_n>0$.



Let us now see how we can use \eqref{8} along with the microlocal Kakeya-Nikodym estimate \eqref{7} to prove Proposition~\ref{prop3}.

\begin{proof}[Proof of Proposition~\ref{prop3}]

In this proof, let us express the microlocal log-power gain in \eqref{7} as
\begin{equation}\label{22}
\e(\la)=(\log \la)^{-\frac1{q_c}}.
\end{equation}

With $h=\rho_\la f$, we have from \eqref{8} that for $n\ge2$
\begin{equation}\label{11}
\|\tilde \rho_\la f||_{L^{q_c}(A_-)} \le C\la^{\frac1{q_c}} (\log\la)^{-1}\|f\|_{L^2(M)}
+C\Bigl(\sum_\nu \| \tilde \sigma_\la \qn \rho_\la f\|_{L^{q_c}(M)}^{q_c}\Bigr)^{1/q_c}.
\end{equation}
Note that the universal (local) spectral projection bounds from \cite{sogge88} and \eqref{i.b3}  yield
\begin{equation}\label{blah}
\|\tilde \sigma_\la\|_{L^2\to L^{q_c}}=O(\la^{\frac1{q_c}}).
\end{equation}
Using this and the fact that the $\qn$ are almost orthogonal, we find that
\begin{equation}\label{12}
\sum_\nu  \| \tilde \sigma_\la \qn \rho_\la f\|_{L^{q_c}(M)}^{2}
\lesssim \la^{\frac2{q_c}} \sum_\nu \|\qn \rho_\la f\|_2^2
\lesssim \la^{\frac2{q_c}}\|\rho_\la f\|_2^2
\lesssim \la^{\frac2{q_c}}\| f\|_2^2.
\end{equation}

Next, note that by \eqref{commute} and the fact that $\|\rho_\la\|_{L^2\to L^2}=O(1)$, we have
\begin{equation}\label{star1}
\|\tilde \sigma_\la Q_\nu \rho_\la f\|_{L^{q_c}(M)}\lesssim
\|Q_\nu\tilde \sigma_\la \rho_\la f\|_{L^{q_c}(M)}+\la^{-\frac14}\|f\|_{L^2(M)}.
\end{equation}
Additionally, $Q_\nu \tilde \sigma_\la =Q_\nu B\sigma_\la =
BQ_\nu\sigma_\la +[Q_\nu,B]\sigma_\la$,  and since the commutator $[Q_\nu,B]$ is in
$S^{-3/4}_{7/8,1/8}$ and band limited (see \cite[Theorem 4.4, Chapter II]{TaylorPDO})
we have $\|[Q_\nu,B]\|_{L^{q_c}\to L^{q_c}}=O(\la^{-\frac34})$.  So, by \eqref{blah} and \eqref{i.b3}
\begin{align}\label{star2}
\|Q_\nu \tilde \sigma_\la \rho_\la f\|_{L^{q_c}(M)}
&\lesssim \|BQ_\nu \sigma_\la \rho_\la f\|_{L^{q_c}(M)}
+\la^{\frac1{q_c}-\frac34}\|f\|_{L^2(M)}
\\
&\lesssim \|Q_\nu \sigma_\la \rho_\la f\|_{L^{q_c}(M)}
+\la^{\frac1{q_c}-\frac34}\|f\|_{L^2(M)}.
\notag
\end{align}
If we combine \eqref{star1} and \eqref{star2}, since the microlocal cutoffs $Q_\nu$ are uniformly
bounded on all $L^p$ spaces, we conclude that 
\begin{align*}
\|\tilde \sigma_\la Q_\nu \rho_\la f\|_{L^{q_c}(M)} &\lesssim
\|Q_\nu \sigma_\la \rho_\la f\|_{L^{q_c}(M)}+\la^{\frac1{q_c}-\frac14}\|f\|_{L^2(M)}
\\
&\lesssim \|Q_\nu \rho_\la f\|_{L^{q_c}(M)}+\|Q_\nu\circ(I-\sigma_\la)\rho_\la f\|_{L^{q_c}(M)}+\la^{\frac1{q_c}-\frac14}\|f\|_{L^2(M)}
\\
&\lesssim \|Q_\nu \rho_\la f\|_{L^{q_c}(M)}+\|(I-\sigma_\la)\rho_\la f\|_{L^{q_c}(M)}+\la^{\frac1{q_c}-\frac14}\|f\|_{L^2(M)}.
\end{align*}
If we use the first part of \eqref{tilde} and our Kakeya-Nikodym bounds \eqref{7} to estimate the first two terms in the right
we conclude that
\begin{equation}\label{13}
\|\tilde \sigma_\la Q_\nu \rho_\la f\|_{L^{q_c}(M)}\lesssim \la^{\frac1{q_c}} \, \e(\la)\, \|f\|_{L^2(M)},
\end{equation}
with $\e(\la)$ as in \eqref{22}.



If we combine \eqref{11}, \eqref{12} and \eqref{13}, we obtain
\begin{align*}
\|\tilde \rho_\la f\|_{L^{q_c}(A_-)} &\lesssim 
\sup_\nu \|\tilde \sigma_\la Q_\nu \rho_\la f\|_{L^{q_c}(M)}^{\frac{q_c-2}{q_c}} \cdot 
\bigl(\, \sum_\nu \|\tilde \sigma_\la Q_\nu \rho_\la f\|_{L^{q_c}(M)}^2\, \bigr)^{\frac1{q_c}}
+\la^{\frac1{q_c}}(\log\la)^{-1} \|f\|_{L^2(M)}
\\
&\lesssim \bigl(\la^{\frac1{q_c}} \e(\la))^{\frac{q_c-2}{q_c}} \, (\la^{\frac2{q_c}})^{\frac1{q_c}}\|f\|_{L^2(M)}+\la^{\frac1{q_c}}(\log\la)^{-1}\|f\|_{L^2(M)}
\\&\lesssim \la^{\frac1{q_c}} \, \bigl(\e(\la)\bigr)^{\frac{q_c-2}{q_c}}\|f\|_2
\\
&=\la^{\frac1{q_c}} \bigl(\e(\la)\bigr)^{\frac2{n+1}} \, \|f\|_2.
\end{align*}
This combined with \eqref{3} yields
\begin{equation}\label{15}
\|\tilde \rho_\la f\|_{L^{q_c}(M)} \lesssim \la^{\frac1{q_c}} (\e(\la))^{\frac2{n+1}} \, \|f\|_{L^2(M)},
\end{equation}
since $(\e(\la))^{\frac2{n+1}}>(\log\la)^{-1/2}$.  This gives the bounds \eqref{i.4} in Theorem~\ref{thm1}  if one recalls \eqref{0.1}.
\end{proof}

\noindent{\bf Proving the bilinear estimates}

Let us prove \eqref{8} by using  estimates from \cite{SBLog} along with \eqref{far}.

As in these inequalities $\tilde \rho_\la f=\tilde \sigma_\la h$, with $h=\rho_\la f$.  So, we need to control
$$\|\tilde \rho_\la f\|_{L^{q_c}(A_-)} = 
\Bigl(\, \int_{\{x: \, |\tilde \sigma_\la h(x)|>\la^{\frac{n-1}4+\frac18}\}}\bigl|(\tilde \sigma_\la h)^2\bigr|^{q_c/2}\, dx \,  \Bigr)^{1/q_c}.
$$

If we recall \eqref{add''} then of course we have for $q\in (\tfrac{2(n+2)}n, q_c)$ as in \eqref{far}
$$|\tilde \sigma_\la h \,  \tilde \sigma_\la h|^{q_c/2}
\le 2^{q/2} \, |\tilde \sigma_\la h \, \tilde \sigma_\la h|^{\frac{q_c-q}2}\, 
\bigl( \, |\diag(h)|^{q/2}+|\far(h)|^{q/2}\, \bigr).$$

Thus,
\begin{align}\label{3.2}
\|\tilde \rho_\la &f\|_{L^{q_c}(A_-)}^{q_c}
=\int_{A_-}|\tilde \sigma_\la h \, \tilde \sigma_\la h|^{q_c/2}\, dx
\\ 
&\lesssim \int_{A_-}|\tilde \sigma_\la h \, \tilde \sigma_\la h|^{\frac{q_c-q}2} \, |\diag(h)|^{q/2} 
+ \int_{A_-}|\tilde \sigma_\la h \, \tilde \sigma_\la h|^{\frac{q_c-q}2}| \, \far(h)|^{q/2} \notag
\\
&=I+II. \notag
\end{align}

To estimate the second term we use \eqref{far}, the ceiling for $A_-$ and the fact that $\|f\|_2=1$ to see that 
\begin{multline*}
II \lesssim \|\tilde \sigma_\la h\|_{L^\infty(A_-)}^{q_c-q} \,  \la^{-(q_c-q)(\frac78 \frac{n-1}2)} \cdot \la
\le \bigl( \, \la^{\frac{n-1}4+\frac18}\, \bigr)^{(q_c-q)} \cdot \la^{-(q_c-q)(\frac78 \frac{n-1}2)} \cdot \la
\\
=\la^{-(q_c-q)[\frac{3(n-1)}{16}-\frac18]} \cdot \la = \lambda^{1-\delta_n}, \quad \delta_n>0,
\end{multline*}
since $(q_c-q) \cdot( \tfrac{3(n-1)}{16}-\tfrac18) >0$.

To estimate $I$, we use the fact that
by Lemma 4.2 in \cite{SBLog}, 
$$\|\diag(h) \|_{L^{{q_c}/2}(M)}\lesssim 
\Bigl(\, \sum_\nu \| \tilde \sigma_\la Q_\nu h\|_{L^{q_c}(M)}^{2q^*}\, \Bigr)^{\frac1{q^*}} \, + \, \la^{-N},$$
where $q^*=q_c/2$ if $n\ge 3$ and $q^*=3/2$ if $n=2$.
Thus, by H\"older's inequality followed by Young's inequality
\begin{align*}
I&\le \| \tilde \sigma_\la h \, \tilde \sigma_\la h\|_{L^{q_c/2}(A_-)}^{\frac{q_c-q}2} \cdot \|\diag\|_{L^{q_c/2}(M)}^{q/2}
\\
&\le \tfrac{q_c-q}{q_c}  \| \tilde \sigma_\la h \, \tilde \sigma_\la h\|_{L^{q_c/2}(A_-)}^{q_c/2} + \tfrac{q}{q_c} \|\diag\|_{L^{q_c/2}(M)}^{q_c/2}
\\
&\le \tfrac{q_c-q}{q_c}  \| \tilde \sigma_\la h \,  \tilde \sigma_\la h\|_{L^{q_c/2}(A_-)}^{q_c/2} +
\Bigl(\, \sum_\nu  \, \|\tilde \sigma_\la  Q_\nu h\|_{L^{q_c}(M)}^{2q^*}\, \Bigr)^{\frac1{q^*} \frac{q_c}2} \, + \, \la^{-N}.
\end{align*}
Since $\tfrac{q_c-q}{q_c}<1$, the first term in the right can be absorbed in the left side of \eqref{3.2}.

Thus, our bounds imply that for $n\ge 3$
\begin{equation}\label{bilinear 3}
\|\tilde \rho_\la f\|_{L^{q_c}(A_-)}
\lesssim 
\Bigl(\, \sum_\nu \| \tilde \sigma_\la Q_\nu h\|_{L^{q_c}(M)}^{2q^*}\, \Bigr)^{\frac1{2q^*}} + \la^{\frac1{q_c}-\delta_n/q_c} \|f\|_{L^2(M)},
\end{equation}
and for $n=2$
\begin{equation}\label{bilinear 2}
\|\tilde \rho_\la f\|_{L^{q_c}(A_-)}
\lesssim 
\Bigl(\, \sum_\nu \| \tilde \sigma_\la Q_\nu h\|_{L^{q_c}(M)}^{3}\, \Bigr)^{\frac1{3}} + \la^{\frac1{q_c}-\delta_n/q_c} \|f\|_{L^2(M)}.
\end{equation}

Note that the first term on the right side of \eqref{bilinear 2} is larger than the one appearing in \eqref{8}
due to the inclusion $\ell^6\subset \ell^3$.  In
the next subsection we shall see how we can modify the arguments we used for term $I$ to get an improved bilinear type inequality rather than \eqref{bilinear 2} when $n=2$.

\noindent{\bf Modified  bilinear arguments for $n=2$}

To be able to have a variant of \eqref{bilinear 2} which involves
an $\ell^6$ sum instead of $\ell^3$, we need to change the definition of $Q_\nu$ a bit. When $n=2$, instead of $\frac 18$, we fix a small $\e_0$, which will be specified later, and
let $\nu$ denote a $\la^{-\e_0}$--separated
set in $S^{1}$ lying in our conic neighborhood
of $\xi_0=(0,\dots,0,1)$.  We then choose a
partition of unity $\sum\beta_\nu(\xi)=1$ on this set
which includes a neighborhood of the $\xi$-support of $B(x,\xi)$
so that each $\beta_\nu$ is supported in a $2\la^{-\e_0}$ cap about $\nu\in S^{1}$, and so that if 
$\beta_\nu(\xi)$ is the homogeneous of degree zero
extension of $\beta_\nu$ to $\Rn\backslash 0$ we have
$$|D^\alpha \beta_\nu(\xi)|\le C_\alpha
\la^{|\alpha|\e_0} \quad \text{if } \, \,
|\xi|=1.$$
If we let $q_\nu(x,\xi)$ then be defined
as in \eqref{qnusymbol},
it follows that the pseudo-differential operators
$Q_\nu$ with symbols belong to a bounded subset
of $S^0_{1-\e_0,\e_0}(M)$. And similar to \eqref{commute}, we can use the arguments in \cite{SBLog}, which involve applications of Egorov's theoerem and the universal bound in \cite{sogge88} to see that 
\begin{equation}\label{commutee0} 
\bigl\|  \tilde \sigma_\la Q_\nu -Q_\nu 
\tilde \sigma_\la 
\bigr\|_{L^2(M)\to L^{q_c}(M)} \lesssim
\la^{\frac1{q_c}-\frac12+2\e_0},
\end{equation}
which is better than the estimate in \eqref{commute} if $\e_0$ is small since $q_c=6$ when $n=2$.

It is straightforward to check that \eqref{add} and \eqref{alor} hold for the $Q_\nu$ operators defined as above, and by repeating the previous arguments, we also have the analog of \eqref{7}, as long as we let $T=c_0 \log\la$ with $c_0\ll \e_0$. These estimates ensure that the arguments \eqref{22}-\eqref{15}, which is the proof of Proposition~\ref{prop3}, still work in our current setting.

We are choosing the $\la^{-\e_0}$ scale instead of $\la^{-\frac18}$,  since in our later arguments, we shall make explicit use of the fact that the number of choices of $\nu$ is small.   See the remark below \eqref{b25'} for more details. 

Write 
$$(\tilde \rho_\la h)^{2}= \, 
\diag(h)+ \far(h) ,
$$
where $\far(h)$ is defined as in \eqref{ndsum}, and
\begin{multline}\label{diage0}
\diag(h)=\sum_{(\nu',\tilde \nu')\in
\Xi_J}
\bigl(\tilde \sigma_\la Q_\nu h\bigr) \cdot
\bigl(\tilde \sigma_\la Q_{\tilde \nu}h\bigr), 
\\
 \text{if } \, \, 
\la^{-\e_0}\in (2^{J-4}, 2^{J-3}], \, \, \text{and } \quad h=\rho_\la f.
\end{multline}
Here as before, $\nu=(\nu', \nu_n)$, and the  definition of $\Xi_J$ is the same as the one in \eqref{org},  which include diagonal pairs where
$\nu'=\tilde \nu'$, and all $(\nu',\tilde \nu')\in
\Xi_J$ satisfy $|\nu'-\tilde \nu'|\le 32\la^{-\e_0}$.
Thus, for each fixed $\nu'$, we have
$\#\{\tilde \nu': \, (\nu', \tilde \nu')\in
\Xi_J\}=O(1)$.

By \eqref{3.2}, to prove improved version of \eqref{bilinear 2}, it suffices to control the term $I$ involving $\diag(h)$. The other term $II$ involving $\far(h)$ can be handled like before and it actually satisfies better bound since we have improved bilinear estimates with the wider angular separation for the off-diagonal terms here.
To proceed, let us define $$T_\nu h=\sum_{\tilde\nu': \,(\nu',\tilde \nu')\in
\Xi_J}(\tilde\sigma_\la Q_\nu h)(\tilde\sigma_\la Q_{\tilde\nu} h),$$
and write 
\begin{equation}\label{b25}
\begin{aligned}
    ( \diag(h))^{2} &=\big(\sum_\nu T_\nu h\big)^2
    \\&= \sum_{|\nu_1'-\nu_2'|\ge 128\la^{-\epsilon_0}}  T_{\nu_1} hT_{\nu_2} h
+\sum_{|\nu_1'-\nu_2'|\le 128\la^{-\epsilon_0}}  T_{\nu_1} hT_{\nu_2} h  \\
    &=A+B.
    \end{aligned}
\end{equation}
Then for $q\in (4,6)$, we have
\begin{equation}\label{b25'}
\begin{aligned}
    | \, \diag(h)|^{q/2}& \lesssim \big| \sum_{|\nu_1'-\nu_2'|\ge 128\la^{-\epsilon_0}}  T_{\nu_1} hT_{\nu_2} h\big|^{\frac q4}
    \\
    &\qquad\qquad+ \big| \sum_{|\nu_1'-\nu_2'|\le 128\la^{-\epsilon_0}}  T_{\nu_1} hT_{\nu_2} h \big|^{\frac q4} \\
    &=|A|^{\frac q4}+|B|^{\frac q4}.
    \end{aligned}
\end{equation}

By splitting $( \diag(h))^{2}$ as a sum of $A$ and $B$ as above, we essentially want to make use of the bilinear estimates for a second time, where $A$ corresponds to the off-diagonal terms and $B$ corresponds to the diagonal ones. However, for the off-diagonal terms, here we are not using the classical Whitney-type decomposition as in \eqref{add'}-\eqref{org} since the $T_\nu$ operators themselves are defined as product of the $Q_\nu$ operators, which makes it difficult to define the $\sum_\nu T_\nu$ operators for certain collection of nearby $\nu$ indexes. Instead,  we assume that the number of choices of $\nu$ is small, which can be achieved by choosing $\e_0$ to be small enough, so that after applying bilinear estimates for each of the single term in $A$, we can sum them up in a rather crude but still sufficient way. 
We shall also use the fact that when $n=2$, $q_c=6$ and  thus $\frac{6}{4}\in (1,2)$, which allows us to follow the strategies in \cite{SBLog} to control the diagonal term $B$ using an almost orthogonal type inequality(see \eqref{almost orthogonal} below).

More explicitly, to estimate the term involving $A$ in $I$, let us first assume that $\nu_1, \nu_2$ are fixed, with $|\nu_1'-\nu_2'|\ge 128\la^{-\e_0}$. By
the  Cauchy-Schwarz inequality and the remarks below \eqref{diage0}, it is easy to see that
$$|T_{\nu_1} hT_{\nu_2} h|\le C\big(\sum_{\tilde\nu_1': \,|\tilde\nu_1'-\nu_1'|\le 32\la^{-\e_0}}|\tilde\sigma_\la Q_{\tilde\nu_1} h|^2\big)
\big(\sum_{\tilde\nu_2': \,|\tilde\nu_2'-\nu_2'|\le 32\la^{-\e_0}}|\tilde\sigma_\la Q_{\tilde\nu_2} h|^2\big)
$$
Since the number of terms in both summations are 
bounded independently of $\la$,
it suffices to control a single term involving
$|\tilde\sigma_\la Q_{\tilde\nu_1} h\tilde\sigma_\la Q_{\tilde\nu_2} h|^2,
$
where $|\tilde\nu_1'-\tilde\nu_2'|\approx 2^j\in (32\la^{-\epsilon_0},1)$.
To proceed, let $\mu_1,\mu_2\in {S}^{1}$  satisfy $|\mu_1-\mu_2|\approx 2^j$, and each be associated with an operator $Q_{\mu_1}$, $Q_{\mu_2}$, where 
$Q_{\mu_1}$ is the pseudo-differential operator with symbol $q_{\mu_1}(x,\xi)$, and we assume $q_{\mu_1}(x,\xi)$ is supported in a $2^j$ conical neighborhood of $\mu$, with $q_{\mu_1}(x,\xi)\equiv 1$ in a neighborhood of the support of $Q_{\tilde\nu_1}$. It is not hard to see that
$$\|(I-Q_{\mu_1})\circ Q_{\tilde \nu_1}\|_{L^2\rightarrow L^2}\lesssim_N \la^{-N},\,\,\forall\,\,\,N>0.
$$
Thus, if we we define $Q_{\mu_2}$ similarly, we have
$$|\tilde\sigma_\la Q_{\tilde\nu_1} h\tilde\sigma_\la Q_{\tilde\nu_2} h|^2=|\tilde\sigma_\la Q_{\mu_1} Q_{\tilde\nu_1} h\tilde\sigma_\la Q_{\mu_2}Q_{\tilde\nu_2} h|^2+
O(\la^{-N}).
$$
Now if we apply bilinear estimates at the scale $2^j$ in $L^{\frac q2}$(see e.g., (4.12) in \cite{SBLog}), along with the fact that $\tilde\sigma_\la h $ is bounded by $\la^{\frac{n-1}4+\frac18}$ (with $n=2$) on the set $A_-$,
we  have, similar to before, 
\begin{equation}\label{b24}
\begin{aligned}
    \int_{A_-}|\tilde\sigma_\la h \, \tilde\sigma_\la h|^{\frac{6-q}2}| (\tilde\sigma_\la Q_{\mu_1} Q_{\tilde\nu_1} h)^2(\tilde\sigma_\la Q_{\mu_2}Q_{\tilde\nu_2} h)^2|^{q/4} \le \la^{1-\delta_2},
    \end{aligned}
\end{equation}
for some fixed $\delta_2>0$. Here the $\la^{-\delta_2}$ gain is independent of $\e_0$ and $j$, whenever we have $2^j\ge \la^{-\frac18}$.

Since the number of choices of $\nu_1', \nu_2'$ in the sum in $A$ is bounded by $\la^{2\e_0}$, \eqref{b24} implies
\begin{equation}
\begin{aligned}
    \int_{A_-}|\tilde\sigma_\la h \, \tilde\sigma_\la h|^{\frac{6-q}2}|A|^{q/4} \lesssim \la^{1+3\epsilon_0-\delta_2}+\la^{-N}.
    \end{aligned}
\end{equation}
By choosing $\epsilon_0$ small enough so that $3\epsilon_0 \le \delta_2/2$, this gives us the desired bound.

To control the term involving $B$, let us define 
\begin{equation}
    \begin{aligned}
    S_{\nu_1} h &=\sum_{\nu_2': |\nu_1'-\nu_2'|\le 128\la^{-\epsilon}}  T_{\nu_1} hT_{\nu_2} h \\
    &=\sum_{\tilde\nu_1',\nu_2',\tilde\nu_2'}(\tilde\sigma_\la Q_{\nu_1} h)(\tilde\sigma_\la Q_{\tilde\nu_1} h)(\tilde\sigma_\la Q_{\nu_2} h)(\tilde\sigma_\la Q_{\tilde\nu_2} h),
    \end{aligned}
\end{equation}
where $|\tilde\nu_1'-\nu_1'|\le 32\la^{-\e_0}$, $|\tilde\nu_2'-\nu_2'|\le 32\la^{-\e_0}$, $|\nu_1'-\nu_2'|\le 128\la^{-\e_0}$, and thus $|\tilde\nu_2'-\nu_1'|\le 160\la^{-\e_0}$. Overall, the number of terms in the second line is 
bounded independently of $\la$,
and we have $B=\sum_{\nu_1} S_{\nu_1} h$. If we use a similar argument as in the proof of Lemma 4.2 in \cite{SBLog}, it is not hard to show the following almost orthogonal type inequality  
\begin{equation}\label{almost orthogonal}
    \begin{aligned}
    \|\sum_{\nu_1} S_{\nu_1} h\|_{L^{{3}/2}(M)}
    \lesssim 
\Bigl(\, \sum_{\nu_1} \| S_{\nu_1} h\|^{\frac32}_{L^{{3}/2}(M)}\, \Bigr)^{\frac2{3}} \, + \, \la^{-N},
    \end{aligned}
\end{equation}
which implies that 
\begin{equation}\label{almost orthogonal 1}
    \begin{aligned}
    \|\sum_{\nu_1} &S_{\nu_1} h\|_{L^{{3}/2}(M)}
  \\
    &\lesssim 
\Bigl(\, \sum_{\nu_1} \|\sum_{\tilde\nu_1',\nu_2',\tilde\nu_2'}(\tilde\sigma_\la Q_{\nu_1} h)(\tilde\sigma_\la Q_{\tilde\nu_1} h)(\tilde\sigma_\la Q_{\nu_2} h)(\tilde\sigma_\la Q_{\tilde\nu_2} h)\|_{L^{{3}/2}(M)}^{\frac32}\, \Bigr)^{\frac2{3}} \, + \, \la^{-N}
     \\
    &\lesssim 
\Bigl(\, \sum_{\nu_1} \| \tilde \sigma_\la Q_{\nu_1} h\|_{L^{6}(M)}^{6}\, \Bigr)^{\frac2{3}} \, + \, \la^{-N},
    \end{aligned}
\end{equation}
where in the second inequality we essentially used the fact that for fixed $\nu_1$, the number of terms in the sum  is finite.

Thus, by H\"older's inequality followed by Young's inequality
\begin{align*}
 \int_{A_-}|\tilde\sigma_\la h \,\tilde \sigma_\la h|^{\frac{6-q}2}|B|^{q/4} &\le \|\tilde\sigma_\la h \,  \tilde\sigma_\la h\|_{L^{3}(A_-)}^{\frac{6-q}2} \cdot \|B\|_{L^{3/2}(M)}^{q/4}
\\
&\le \tfrac{6-q}{6}  \|\tilde\sigma_\la h \,  \tilde\sigma_\la h\|_{L^{3}(A_-)}^{3} + \tfrac{q}{6} \|B\|_{L^{3/2}(M)}^{3/2}
\\
&\le \tfrac{6-q}{6}  \|\tilde\sigma_\la h \,  \tilde\sigma_\la h \|_{L^{3}(A_-)}^{3} +
 \sum_{\nu_1}  \, \|\tilde\sigma_\la  Q_{\nu_1} h\|_{L^{6}(M)}^{6}\, \, + \, \la^{-N}.
\end{align*}
Since $\tfrac{6-q}{6}<1$, the first term in the right can be absorbed in the left side of \eqref{3.2}.

Thus, the proof of \eqref{8} for the case $n= 2$ is complete.

\newsection{Further improvements for negatively curved manifolds}\label{neg}
As we shall see in the next subsection, if we assume that {\em all} of the sectional curvatures of $(M,g)$ are {\em negative}, then
we can improve the second
global estimate, \eqref{7}, that we used to
\begin{equation}\label{77}
\sup_\nu \|Q_\nu \rho_\la\|_{L^2(M)\to L^{q_c}(M)}\lesssim \la^{\frac1{q_c}} \cdot \e(\la),
\quad \text{with } \, \, \e(\la)=(\log \la)^{-1/2}.
\end{equation}

If one repeats the above arguments, one finds that this implies that we can therefore improve Theorem~\ref{thm1} under this curvature assumption as follows.

\begin{theorem}\label{thm2}
Assume that $(M,g)$ is an $n$-dimensional compact manifold with all sectional curvatures being negative.  Then, if $T=c_0\log\la$ with $c_0>0$ sufficiently small, we have
\begin{equation}\label{100}
\|\rho(T(\la-P))f\|_{\lqcm} \lesssim \la^{1/q_c} \, \bigl(\log\la\bigr)^{-\alpha_n}\|f\|_2,
\end{equation}
where $\alpha_n=\tfrac{1}{(n+1)}$.
\end{theorem}

Note that when $n=2$, $q_c=6$, and so \eqref{100} says that
\begin{equation}\label{116}
\|\rho_\la f\|_{L^6(M)} \lesssim \la^{\frac16} \, (\log\la)^{-\frac13} \, \|f\|_{L^2(M)}, \quad n=2,
\end{equation}
while
when $n=3$, $q_c=4$, and so \eqref{100} says that
\begin{equation}\label{117}
\|\rho_\la f\|_{L^{4}(M)} \lesssim \la^{\frac1{4}} (\log\la )^{-\frac1{4}} \, \|f\|_{L^2(M)}, \quad n=3.
\end{equation}

Estimates \eqref{116} and \eqref{117} 
should be compared to bounds in Hickman~\cite{Hickman} and Germain and Myerson~\cite{Germain} who extended the toral
eigenfunction estimates of Bourgain and Demeter~\cite{BourgainDemeterDecouple}.
As we shall see in \S\ref{g}, the three-dimensional estimates \eqref{117} would be optimal if $M$ were a product manifold with
$S^1$ as a factor; however, in this case, $M$ cannot have all sectional curvatures negative, since some would have to equal zero.
Perhaps not unexpectedly, in two-dimensions, the bounds in \eqref{116} are stronger than the sharp estimates that we  have
obtained for two-dimensional tori for spectral windows of width $(\log\la)^{-1}$.  It would be interesting to determine whether or not
this phenomena persists in higher dimensions.


\noindent{\bf Proof of $L^2\to L^{q_c}$ Kakeya-Nikodym estimates assuming negative curvature}

The main step in the proof of \eqref{77} will be to establish certain pointwise estimates using an argument that is almost identical
one in \cite{BlairSoggeToponogov}.  Basically, the only difference is that, here, we shall exploit the fact that, if all the sectional curvatures of
$(M,g)$ are negative, then the leading term in the Hadamard parametrix for the solution of the wave equation in the universal
cover of $(M,g)$ decays exponentially in terms of the time parameter.  This fact was not used by two of us in \cite{BlairSoggeToponogov}
or its predecessor, Sogge and Zelditch \cite{SoggeZelditchL4},  since we were interested in results for general manifolds of non-positive
curvature and the techniques in these papers or \cite{SBLog} only provided somewhat modest improvements for spectral projection estimates.

Turning to the proof of \eqref{77}, just as in \cite{BlairSoggeToponogov}, we shall want to make use of the fact that if $Q_\nu(x,y)$ denotes the
kernels of our microlocal cutoffs there, then we have the uniform bounds
\begin{equation}\label{n.1} \sup_x \int|Q_\nu(x,y)|\, dy \le C.
\end{equation}
This implies that $Q_\nu: \, L^\infty(M)\to L^\infty(M)$ uniformly, and, since they are also uniformly bounded on $L^2(M)$, by interpolation and duality,
we conclude that
\begin{multline}\label{n.2}
\|Q_\nu\|_{L^p(M)\to L^p(M)}=O(1), \quad 2\le p\le \infty \\ \text{and } \, \, 
\|Q^*_\nu\|_{L^p(M)\to L^p(M)}=O(1), \quad 1\le p\le 2.
\end{multline}

Also, of course if, as in \eqref{14}, $\Psi=\rho^2$, then \eqref{77} is equivalent to the estimate
\begin{equation}\label{n.3}
\bigl\|\, Q_\nu \Psi(T(\la-P))Q^*_\nu\, \bigr\|_{L^{q_c'}(M)\to L^{q_c}(M)} \lesssim \la^{2/q_c} \, (\log\la)^{-1},
\end{equation}
if, as before,
\begin{equation}\label{n.4}
T=c_0\log\la,
\end{equation}
with $c_0>0$ sufficiently small.

By \eqref{n.2}, this would be a consequence of 
\begin{equation}\label{n.5}
\bigl\| \, Q_\nu \Psi(T(\la-P))\, \bigr\|_{L^{q_c'}(M)\to L^{q_c}(M)}\lesssim \la^{2/q_c} \, (\log\la)^{-1}.
\end{equation}
The proof of this will closely follow the related $L^2\to L^2$ Kakeya-Nikodym estimate (3.6) in \cite{BlairSoggeToponogov}.

First, as was done in the proof of Proposition~\ref{prop2}, let us split our smoothed out spectral projection operators 
$\Psi(T(\la-P))$ into the sum of a ``local'' and ``global'' part.  This time, since we shall want to dyadically decompose
the ``global'' piece, let us assume, as before, that our Littlewood-Paley function $\beta\in C^\infty_0((1/2,2))$ satisfies
$\sum_{j=-\infty}^\infty \beta(t/2^j)=1$, $t>0$.  We then let
$$\beta_0(t)=1-\sum_{j=0}^\infty \beta(j/2^j)\in C^\infty_0({\mathbb R}_+),$$
and note that $\beta_0(t)\equiv 1$ for $t>0$ near the origin.  The ``local'' operator then is
$$L_\la =\frac1{2\pi T} \int e^{i\la t} e^{-itP} \beta_0(|t|) \, \Hat \Psi(t/T) \, dt,$$
and we have
$$\|Q_\nu L_\la\|_{L^{q'_c}(M)\to L^{q_c}(M)}=O(\la^{2/q_c} \, (\log\la)^{-1})$$
using \eqref{n.2}, \eqref{n.3} and the universal local spectral projection estimates in \cite{sogge88} just as we did
in the proof of Proposition~\ref{prop2}.

As a result of these local bounds, we would have \eqref{n.5} and be done if we could show that
\begin{equation}\label{n.6}
\|Q_\nu G_\la\|_{L^{q'_c}(M)\to L^{q_c}(M)} =O(\la^{2/q_c}\, (\log\la)^{-1}),
\end{equation}
if
\begin{equation}\label{n.7}
G_\la =\frac1{2\pi T}\int e^{i\la t} e^{-itP} \, \bigl(1-\beta_0(|t|)\bigr) \, \Hat \Psi(t/T) \, dt.
\end{equation}
Note that if we replace $e^{-itP}$ by $e^{itP}$ here, then the resulting operator has a kernel which
is $O(\la^{-N})$ and so trivially enjoys the bounds in \eqref{n.6}.

Consequently, by Euler's formula, it suffices to show that if
\begin{equation}\label{n.8}
\tilde G_\la =\frac1{\pi T}\int e^{i\la t} \cos(t\sqrt{-\Delta_g}) \, \bigl(1-\beta_0(|t|)\bigr) \, \Hat \Psi(t/T) \, dt,
\end{equation}
then
\begin{equation}\label{n.9}
\|Q_\nu \tilde G_\la\|_{L^{q'_c}(M)\to L^{q_c}(M)}=O(\la^{2/q_c} \, (\log\la)^{-1}).
\end{equation}
If we let for $\mu=2^j\ge1$
\begin{equation}\label{n.88}
\tilde G_{\la,\mu} =\frac1{\pi T}\int e^{i\la t} \cos(t\sqrt{-\Delta_g}) \, \beta(|t|/\mu) \, \Hat \Psi(t/T) \, dt,
\end{equation}
then since
$$\tilde G_\la =\sum_{\{\mu=2^j: \, j\ge0\}} \tilde G_{\la,\mu},$$
in order to prove \eqref{n.9}, it suffices to show that
\begin{equation}\label{n.10}
\|Q_\nu \tilde G_{\la,\mu}\|_{L^{q'_c}(M)\to L^{q_c}(M)}\le C\la^{2/q_c} \, (\log\la)^{-1} \, \mu^{-1}.
\end{equation}
We note that since $\Hat \Psi$ is compactly supported, by Huygens' principle, $\tilde G_{\la,\mu}=0$ if $\mu$ is larger than
a fixed multiple of $T$ as in \eqref{n.4}. 

As before, we shall prove \eqref{n.10} using
interpolation.
To this end, we first note that since the Fourier transform of
$$t\to T^{-1}\beta(|t|/\mu) \Hat \Psi(t/T)$$
is $O(\mu/T)$, by \eqref{n.2} and orthogonality, we have
\begin{equation}\label{n.11}
\|Q_\nu \tilde G_\la\|_{L^{2}(M)\to L^{2}(M)}=O(\mu/T).
\end{equation}

The other estimate that we shall use, which requires our curvature assumption, is that
\begin{equation}\label{n.12}
\|Q_\nu \tilde G_\la\|_{L^{1}(M)\to L^{\infty}(M)}\le C_{K,N}T^{-1}\la^{\frac{n-1}2} \mu^{-N}, \quad N=1,2,\dots,
\end{equation}
if all of the sectional curvatures of $(M,g)$ are assumed to be $\le -K^2<0$.  This is analogous to
(3.8$'$) in \cite{BlairSoggeToponogov} which did not include the rapid decay $\mu^{-N}$ since the
curvatures there were merely assumed to be non-positive.  Also, there, we did not break things up dyadically as we have done so here.

By a simple interpolation argument (as in the proof of \eqref{7} above), \eqref{n.11} and \eqref{n.12} imply \eqref{n.10}.
So, we have reduced matters to establishing \eqref{n.12}.

Just as in \cite{BlairSoggeToponogov} and other works, we have switched from the half-wave operator in \eqref{n.7} to
the use of $\cos t\sqrt{-\Delta_g}$ so that we can use the Hadamard parametrix and the Cartan-Hadamard theorem to lift the
calculations that will be needed for \eqref{n.12} up to the universal cover $(\Rn, \tilde g)$ of $(M,g)$.  This approach was used 
earlier in \cite{BlairSoggeToponogov} and \cite{SoggeZelditchL4}.

To this end, let $\{\alpha\}=\Gamma$ denote the group of deck transformation preserving the associated 
covering map $\kappa: \, \Rn\to M$ coming from the exponential map at the point in $M$ with coordinates $0$ in $\Omega$ above.
The metric $\tilde g$ on $\Rn$ that we mentioned above is the pullback of $g$ via $\kappa$.  Choose a Dirichlet domain $D\simeq
M$ for $M$ centered at the lift of the point with coordinates $0$.

Following \cite{Berard},  \cite{SoggeZelditchL4} and others, we recall that if $\tilde x$ denotes the lift of $x\in M$ to $D$, then we have the following formula
$$\bigl(\cos t\sqrt{-\Delta_g}\bigr)(x,y)
=\sum_{\alpha\in \Gamma}\bigl(\cos t\sqrt{-\Delta_{\tilde g}}\bigr)(\tilde x, \alpha(\tilde y)),$$
which is a cousin of the classical Poisson summation formula.  As a result, for later use, 
if $K_{\la,\mu}(x,y)$ denotes the kernel of $\tilde G_{\la,\mu}$ then we have the formula
\begin{multline}\label{n.13}
K_{\la,\mu}(x,y)=\sum_{\alpha\in \Gamma}K_{\la,\mu}(\tilde x, \alpha(\tilde y)),
\\
\text{if } \, \, 
K_{\la,\mu}(\tilde x,\tilde y)=\frac1{\pi T} \int e^{i\la t} \, \bigl(\cos t\sqrt{-\Delta_{\tilde g}}\bigr)(\tilde x,\tilde y)
\, \beta(|t|/\mu) \, \Hat \Psi(t/T) \, dt.
\end{multline}
The key estimate for us, which is an improvement of (3.8) in \cite{BlairSoggeToponogov} and requires our
curvature assumption is that
\begin{equation}\label{n.14}
|K_{\la,\mu}(\tilde x,\tilde y)|\le C_{K,N} T^{-1} \, \la^{\frac{n-1}2} \, 
\bigl(d_{\tilde g}(\tilde x,\tilde y)\bigr)^{-\frac{n-1}2} \mu^{-N}, \quad N=1,2,\dots.
\end{equation}

To prove this, we note that this kernel vanishes if $d_{\tilde g}(\tilde x,\tilde y)$ is larger than $2\mu$  by the Huygens principle
since $\beta(|t|/\mu)=0$ for $|t|\ge 2\mu$.  Also, since this function also vanishes for $|t|\le \mu/2$, it is straightforward to use
the Hadamard parametrix to see that $K_{\la,\mu}(\tilde x,\tilde y)=O(\la^{-N})$ for every $N$ if
$d_{\tilde g}(\tilde x,\tilde y)\le \mu/4$, which is better than the bounds in \eqref{n.14} as
$\mu\lesssim \log \la \ll \la$.

Due to these simple facts, to prove \eqref{n.14}, we may assume that $d_{\tilde g}(\tilde x,\tilde y)\approx \mu$.  We shall
make use of  the Hadamard parametrix for $\cos t\sqrt{-\Delta_{\tilde g}}$.  As is well known and described, for instance,  in
\cite{Berard}, \cite{SoggeHangzhou} and \cite{SoggeZelditchL4}, the leading term in this parametrix is of the form
\begin{equation}\label{n.15}
w_0(\tilde x,\tilde y) \cdot (2\pi)^{-n}\int_{\Rn}e^{id_{\tilde g}(\tilde x,\tilde y)\xi_1}
\, \cos(t|\xi|) \, d\xi,
\end{equation}
where in geodesic normal coordinates about $\tilde x$  this leading coefficient is given by
\begin{equation}\label{n.16}
w_0(\tilde x,\tilde y)=\bigl(\, \text{det }\tilde g_{_{ij}}(\tilde y)\, \bigr)^{-1/4}.
\end{equation}
Thus, if in geodesic polar coordinates the volume element is given by
$$dV_{\tilde g}(\tilde y)=\bigl(\, {\cal A}(t,\omega)\, \bigr)^{n-1} \, dt d\omega, \, \, \,
t=d_{\tilde g}(\tilde x,\tilde y),$$
then
$$w_0(\tilde x,\tilde y)=\bigl( \, t/{\mathcal A}(t,\omega)\, \bigr)^{\frac{n-1}2}.$$
By the classical G\"unther comparison theorem from Riemannian geometry (see \cite[\S III.4]{ChavelRiemannianGeometry})
\begin{equation}\label{n.17}
{\mathcal A}(t,\omega)\ge \frac1K \, \sinh(Kt),
\end{equation}
if, as we are assuming all the sectional curvatures are $\le -K^2<0$.
Thus,
\begin{equation}\label{n.18}
w_0(\tilde x,\tilde y)\le C_{K,N} \mu^{-N} \quad \text{if } \, \, \,
d_{\tilde g}(\tilde x,\tilde y)\approx \mu.
\end{equation}

Also, a routine stationary phase calculation yields
\begin{multline}\label{n.19}
T^{-1}\iint e^{i\la t} e^{id_{\tilde g}(\tilde x,\tilde y)\xi_1}
\, \cos(t|\xi|) \,  \beta(|t|/\mu) \, \Hat \Psi(t/T) \, dt d\xi
\\
=O(T^{-1}\la^{\frac{n-1}2} \, (d_{\tilde g}(\tilde x,\tilde y))^{-\frac{n-1}2}),
\quad
\text{if } \, \, d_{\tilde g}(\tilde x,\tilde y) \approx \mu \ge 1.
\end{multline}

Due to \eqref{n.18} and \eqref{n.19}, if we replace $(\cos t\sqrt{-\Delta_{\tilde g}})(\tilde x,\tilde y)$ in
\eqref{n.13} with the leading term \eqref{n.15} in its Hadamard parametrix, the resulting expression will
satisfy the bounds in \eqref{n.14}.  Since one can also use stationary phase to see that the lower order terms in the 
Hadamard parametrix (see, e.g., \cite[Theorem 2.4.1]{SoggeHangzhou}) will each make contributions
which are $O(\la^{\frac{n-1}2-1})$ which is better than the bounds posited in \eqref{n.14},  we obtain this bound.

Unfortunately, although \eqref{n.14} is sufficient for obtaining bounds like \eqref{2}, it, by itself, can not be sufficient for
proving \eqref{n.12} since there are $\exp(c\mu)$ terms in the sum in \eqref{n.13} with $d_{\tilde g}(\tilde x,\tilde y)\approx \mu$.
 To prove \eqref{n.12}, as in \cite{BlairSoggeToponogov}, we need to use the fact that \eqref{n.12} includes the microlocal
cutoff which means that the kernel of $Q_\nu \tilde G_{\la,\mu}$,
\begin{equation}\label{n.20}
\sum_{\alpha\in \Gamma}\frac1{\pi T} \int
e^{i\la t} \, \bigl(Q_\nu \cos t\sqrt{-\Delta_{\tilde g}}\bigr)(\tilde x, \alpha(\tilde y)) \, \beta(|t|/\mu) \, 
\Hat \Psi(t/T)\, dt,
\end{equation}
only involves $O(\mu)$ non-trivial terms.  In the above $Q_\nu$ denotes the pullback of the operator $Q_\nu$ to  the fundamental domain $D$.

To see this we shall argue exactly as in \cite{BlairSoggeToponogov}.  We first recall that through the point $x\in \Omega$ there
is a unique geodesic $\gamma_{x,\nu}$ which passes through the plane $z_n=0$ with covector $\xi_\nu$ at this point in the plane.
After modifying the coordinates in $\Omega$, we may assume that $z'=0$ at the intersection point in the plane.
Then, as in \cite{BlairSoggeToponogov}, we let $\tilde \gamma(t)$, $t\in \R$ denote the lift of the projection of the geodesic
$\gamma$ in $\Omega$ to the universal cover and
$${\mathcal T}_R(\tilde \gamma)=\{\tilde z: \, d_{\tilde g}(\tilde \gamma, \tilde z)\le R\}$$
denote an $R$-tube about $\tilde \gamma$.  Then, just as in \cite{BlairSoggeToponogov}, if $R$ is fixed sufficiently large and
$\alpha(D) \cap {\mathcal T}_R(\tilde \gamma)
=\emptyset$, the corresponding summand in \eqref{n.20} is $O(\la^{-1})$ by Toponogov's theorem
and microlocal arguments.  Indeed, this is exactly how (3.9) in \cite{BlairSoggeToponogov} was proved and one can
simply repeat the arguments there.  Since there are $O(\la^{1/2})$ non-zero terms in \eqref{n.20} 
 if $c_0$ in \eqref{n.4} is fixed small enough, we obtain in this case
$$\sum_{\{\alpha: \, \alpha(D)\cap {\mathcal T}_R(\tilde \gamma)=\emptyset \}}
\frac1{\pi T} \int
e^{i\la t} \, \bigl(Q_\nu \cos t\sqrt{-\Delta_{\tilde g}}\bigr)(\tilde x, \alpha(\tilde y)) \, \beta(|t|/\mu) \, 
\Hat \Psi(t/T)\, dt =O(\la^{-1/2}).
$$
Therefore, the operator with this kernel maps $L^1(M)$ to $L^\infty(M)$ with norm $O(\la^{-1/2})$, which is much better than
the bounds in \eqref{n.12} as $\mu\le \log\la$.

Since \eqref{n.12} is equivalent to the statement that the kernel of $Q_\nu \tilde G_{\la,\mu}$ has sup-norm bounded by the
right side of \eqref{n.12}, it suffices to show that the remaining part of the sum in \eqref{n.20} where $\alpha(D)\cap
 {\mathcal T}_R(\tilde \gamma)\ne \emptyset$ can be bounded by the right side of \eqref{n.12}.  As above, the terms where
$d_{\tilde g}(\tilde x,\alpha(\tilde y))$ is not comparable to $\mu$ yield $O(\la^{-N})$ contributions and so can
be ignored.  Thus, since, by \eqref{n.2}, $Q_\nu: \, L^\infty \to L^\infty$, our task now would be complete if we could show that
\begin{multline}\label{n.21}
 \sum_{\{\alpha: \, \alpha(D)\cap {\mathcal T}_R(\tilde \gamma)\ne \emptyset, \, \text{and } \, d_{\tilde g}(\tilde x,\alpha(\tilde y))\approx \mu\}}
\frac1T \Bigl| \,  \int
e^{i\la t} \, \bigl( \cos t\sqrt{-\Delta_{\tilde g}}\bigr)(\tilde x, \alpha(\tilde y)) \, \beta(|t|/\mu) \, 
\Hat \Psi(t/T)\, dt \, \Bigr|
\\
\le C_{K,N}T^{-1}\, \la^{\frac{n-1}2} \, \mu^{-N}, \quad N=1,2,\dots .
\end{multline}
It is easy to see directly or by arguments in \cite{BlairSoggeToponogov}, this sum only involves $O(\mu)$ terms.  By
\eqref{n.14},  each is of these is $O(T^{-1}\la^{\frac{n-1}2}\mu^{-N})$ for any $N$, which yields \eqref{n.21} and completes
the proof.

\noindent {\bf Remark:}  This argument also yields the
microlocal Kakeya-Nikodym estimate \eqref{7} in which
we were merely assuming that $(M,g)$ has non-positive
curvature.  In this case, the leading coefficient of 
the Hadamard parametrix is no longer rapidly decreasing.  So,
we cannot use \eqref{n.18}.  However, we do have that
$w_0$ is bounded also by the G\"unther comparison
theorem.  If we repeat the above arguments, we would
see that, up to an $O(T^{-1}\la^{\frac{n-1}2}
\mu^{-N})$ error, the kernels $K_{\la,\mu}$ in 
\eqref{14} would be dominated by the left side of
\eqref{n.21}.  Since $w_0$ is bounded, by \eqref{n.19},
each summand is $O(T^{-1}\la^{\frac{n-1}2} \mu^{-\frac{n-1}2})$ since $d_{\tilde g}(\tilde x,\alpha(\tilde y))\approx \mu$ in the sum.  Since there are $O(\mu)$
terms, we conclude that the left side of \eqref{n.21}
under the assumption of non-positive curvature is
$O(T^{-1}\la^{\frac{n-1}2}\mu^{1-\frac{n-1}2})$, and
so we obtain \eqref{14}.

\newsection{Some sharp estimates on tori}\label{torus}

In this section, we shall see how we can modify the previous arguments to get the following sharp improved spectral
projection bounds on  tori.

\begin{theorem}\label{thm3}
Let $n\ge 2$, $\T^n=\R/(\ell_1\mathbb{Z})\times\dots \times \R/(\ell_n\mathbb{Z})$ denote the rectangular torus with $\ell_i\ge1$.
Then,
\begin{equation}\label{4.1}
\|\rho(T(\la-P))f\|_{\lqct} \lesssim \la^{1/q_c} \, T^{-1/q_c}\|f\|_2, \,\,\,\,\,\forall\,\, 1\le T\le \la^{\frac{1}{n+3}-\epsilon}.
\end{equation}
\end{theorem}
Here $\e$ is an arbitrary small constant, and for convenience we are assuming $\ell_i\ge1$, which implies that the injectivity radius of the torus is bounded below by $\frac12$. The implicit constants in \eqref{4.1} may depend on $\ell_i$ and $\e$, but do not depend on $T$ and $\la$. As we shall see later in Proposition~\ref{productmfld},
the  $ T^{-1/q_c}$ gain in the above estimate is sharp. 

To prove \eqref{4.1}, if we define 
$\tilde \rho_\la$ as in \eqref{i.15}, by arguing as before, it suffices to show that
\begin{equation}\label{4.2}
\|\tilde \rho_\la f\|_{L^{q_c}(\T^n)}\le \la^{1/q_c}T^{-1/q_c}\|f\|_2.
\end{equation}

\begin{proposition}\label{prop4.1}  Let 
$$A_+=\{x\in \T^n: \, |\tilde \rho_\la f(x)|\ge C(T\la)^{\frac{n-1}4}\, \} \subset \T^n,$$
where $C$ is a large enough constant which may depend on $\ell_i$.  Then,
\begin{equation}\label{4.3}
\|\tilde \rho_\la f\|_{\lqc(A_+)}\le C\laq \,T^{-1/2} \, \|f\|_{L^2(\T^n)}.
\end{equation}
\end{proposition}

\begin{proof}
We first note that, by \eqref{i.b3} and \eqref{tilde}, we have 
$$\| \tilde\rho_\la f\|_{L^{q_c}(A_+)}\le \|B\rho_\la f\|_{L^{q_c}(A_+)}+C\la^{\frac1{q_c}}/T,$$
since we are assuming that $f$ is $L^2$-normalized.  Thus, we would have \eqref{4.3} if we could show that
\begin{equation}\label{4.3'}
\|B\rho_\la f\|_{L^{q_c}(A_+)}\le C\la^{\frac1{q_c}}T^{-1/2}+\tfrac12 \|\tilde \rho_\la f\|_{L^{q_c}(A_+)}.
\end{equation}

If we repeat the argument in \eqref{4},  by duality,
\begin{align}\label{4.4}
\| B \rho_\la f\|^2_{L^{q_c}(A_+)}&\le\int \big(B\circ L_\la\circ B^*\big)\bigl(\1_{A_+}\cdot g\bigr)(x) \cdot \overline{\1_{A_+}(x) \, g(x)} \, dx \notag
\\
&\qquad \qquad + \int \big(B\circ G_\la\circ B^*\big)\bigl(\1_{A_+}\cdot g\bigr)(x) \cdot \overline{\1_{A_+}(x) \, g(x)} \, dx \notag
\\
&= I +II. \notag
\end{align}
Here, $G_\la$ is the operator whose kernel is in \eqref{2}, while $L_\la$ is the ``local'' operator
$$L_\la =(2\pi T)^{-1} \int a(t) \Hat \Psi(t/T) e^{it\la} e^{-itP}\, dt.$$
As before the local operator $L_\la$ satisfies
$$\|L_\la\|_{L^{q_c'}(\T^n) \to L^{q_c}(\T^n)} \lesssim T^{-1}\la^{2/q_c}.$$
By H\"older's inequality, we have
\begin{equation}\label{4.5}
    \begin{aligned}
    |I| &\le \bigl\|BL_\la B^*\bigr(\1_{A_+}\cdot g\bigr)\bigr\|_{q_c}\cdot\| \1_{A_+} \cdot g\|_{q_c'}
\\
&\lesssim \bigl\|L_\la B^*\bigr(\1_{A_+}\cdot g\bigr)\bigr\|_{q_c}\cdot\| \1_{A_+} \cdot g\|_{q_c'} 
\\
&\lesssim \la^{2/q_c}T^{-1}  \bigl\| B^*\bigr(\1_{A_+}\cdot g\bigr)\bigr\|_{q_c'}\cdot\| \1_{A_+} \cdot g\|_{q_c'} 
\\
&\lesssim \la^{2/q_c}T^{-1}\|g\|^2_{L^{q_c'}(A_+)} 
\\
&=\la^{2/q_c}T^{-1}. 
    \end{aligned}
\end{equation}

To handle the second term, we shall use the fact that  if $\beta\in C^\infty_0((1/2,2))$ and 
\begin{equation}
\label{4.6}
G^j_\la(x,y)=\frac1{2\pi} \int_{-\infty}^\infty \bigl(1-a(t)\bigr)\beta(t/2^j) \, T^{-1} \Hat\Psi(t/T) e^{it\la} \, \bigl(e^{-itP})(x,y)\, dt,
\end{equation}
then  the integral operator with this kernel maps $L^1\to L^\infty$ with norm
$O\bigl(T^{-1} \, 2^{\frac{n+1}2j} \la^{\frac{n-1}2}\bigr)$.
To see this, note that if we replace $e^{-itP}$ above by $e^{itP}$,  the  kernel 
is $O(\la^{-N})$ for all $N>0$. Thus by Euler's formula, it suffices to show that for any fixed $x,y$
\begin{equation}\nonumber
\big| \int_{-\infty}^\infty \bigl(1-a(t)\bigr)\beta(t/2^j) \, T^{-1} \Hat\Psi(t/T) e^{it\la} \, \bigl(\cos(tP))(x,y)\, dt\big|=O\bigl(T^{-1} \, 2^{\frac{n+1}2j} \la^{\frac{n-1}2}\bigr).
\end{equation}
We postpone the proof of this to \eqref{tnn.13}-\eqref{tnn.16} below.

As a result, since the dyadic operators $B$ are uniformly bounded on $L^1$ and $L^\infty$, we can
repeat the argument that we used to estimate $I$ to see that
\begin{equation}\label{4.7}
 \|G_\la\|_{L^{1}(\T^n) \to L^{\infty}(\T^n)}\le \sum_{1\le2^j\le T}
 \|G^j_\la\|_{L^{1}(\T^n) \to L^{\infty}(\T^n)}\lesssim   (T\la)^{\frac{n-1}2}.
\end{equation}
Then \eqref{4.7} yields
\begin{equation*}
|II|\le C_0(T\la)^{\frac{n-1}2} \, \bigl\| \1_{A_+} \cdot g\bigr\|_1^2 \le C_0(T\la)^{\frac{n-1}2} \,
\|g\|^2_{q_c'} \, \|\1_{A_+}\|_{q_c}^2
=C_0(T\la)^{\frac{n-1}2} \, \|\1_{A_+}\|_{q_c}^2.
\end{equation*}
If we recall the definition of $A_+$ and use Chebyshev's inequality we can estimate the last factor:
$$\|\1_{A_+}\|_{q_c}^2 \le \bigl(C(T\la)^{\frac{n-1}4 }\bigr)^{-2}\|\tilde \rho_\la f\|^2_{L^{q_c}(A_+)}.$$
By choosing the constant $C$ large enough such that $C^{-2}C_0\le \frac14$, we have,
$$|II|\le \tfrac14\| \tilde \rho_\la f\|^2_{L^{q_c}(A_+)}.
$$

If we combine this bound with the earlier one, \eqref{4.5}, for $I$ we conclude that
\eqref{4.3'} is valid.
\end{proof}

\begin{proposition}\label{prop4.3} Let 
\begin{equation}\label{4.8}
A_-=\{x\in \T^n: \, |\tilde \rho_\la f(x)|\le C(T\la)^{\frac{n-1}4}\, \} \subset \T^n.\end{equation}
Then, for any fixed $\e>0$ we have
\begin{equation}\label{4.9}
\|\tilde \rho_\la f\|_{\lqc(A_-)}\le C\laq \,T^{-1/q_c} \, \|f\|_{L^2(\T^n)}, \,\,\,\,\,\forall\,\, 1\le T\le \la^{\frac{1}{n+3}-\epsilon}.
\end{equation}
\end{proposition}
In order to make use of the bilinear estimates, we shall 
 work in the microlocalized modes. 
First, let $\nu$ 
range over a 
$T^{-1}$--separated
set in $S^{n-1}$ lying in a neighborhood
of $\xi_0=(0,\dots,0,1)$.  We then choose a
partition of unity $\sum\beta_\nu(\xi)=1$ on this set
which includes a neighborhood of the $\xi$-support of $B(x,\xi)$
so that each $\beta_\nu$ is supported in a $2T^{-1}$ cap about $\nu\in S^{n-1}$, and so that if 
$\beta_\nu(\xi)$ is the homogeneous of degree zero
extension of $\beta_\nu$ to $\Rn\backslash 0$ we have
$$|D^\alpha \beta_\nu(\xi)|\le C_\alpha
T^{|\alpha|} \quad \text{if } \, \,
|\xi|=1.$$
And as in \eqref{qnusymbol}, we define the symbols 
\begin{equation}\label{qnusymboltn}
q_\nu(x,\xi)=\psi(x) \, \beta_\nu\big(\Theta(x, \xi/p(x,\xi)\bigr) \, \tilde\beta\bigl(p(x,\xi)/\la\bigr).
\end{equation}
where 
$\psi \in C^\infty_0(\Omega)$ equals
one in a neighborhood of the $x$-support of $B(x,\xi)$,
and $\tilde \beta\in C^\infty_0((0,\infty))$ equals one in a neighborhood of the support of the Littlewood-Paley bump function
in \eqref{i.b2}. In particular, in a local coordinate chart of $\T^n$, we always have $p(x,\xi)=|\xi|$, which is independent of $x$. Also recall that as in \eqref{map'}, $\Theta(x, \xi/p(x,\xi)\bigr)$ is the third component function of the inverse of the map 
\begin{equation}\label{maptn}
(t,x',\eta)\to \chi_t(x',0,\eta)\in S^*M, \quad
(x',0,\eta)\in S^*M,
\end{equation}
in a neighborhood of $x'=0$, $t=0$
and $\eta=(0,\dots,0,1)$. And on the flat torus $\T^n$, we always have $
 \chi_t(x',0,\eta)=(x'+t\eta',t\eta_n, \eta)
$ for $(t,x',\eta)$ in this small neighborhood, thus  $\Theta(x, \xi/p(x,\xi)\bigr)=\xi/|\xi|$, which is also independent of $x$.
It then follows that, by rescaling, the pseudo-differential operators
$Q_\nu$ with these symbols have Schwartz kernel
\begin{equation}\label{qnu}
    Q_\nu(x,y)=\frac{\la^n}{(2\pi)^n}\int e^{i\la \langle x-y,\xi\rangle}\psi(x) \, \beta_\nu\big(\xi/|\xi|\bigr) \, \tilde\beta(|\xi|)d\xi,
\end{equation}
which satisfies 
\begin{equation}\label{tnn.1} \sup_x \int|Q_\nu(x,y)|\, dy \le C.
\end{equation}
This implies that $Q_\nu: \, L^\infty(M)\to L^\infty(M)$ uniformly, and, since they are also uniformly bounded on $L^2(M)$, by interpolation and duality,
we conclude that
\begin{multline}\label{tnn.2}
\|Q_\nu\|_{L^p(M)\to L^p(M)}=O(1), \quad 2\le p\le \infty \\ \text{and } \, \, 
\|Q^*_\nu\|_{L^p(M)\to L^p(M)}=O(1), \quad 1\le p\le 2.
\end{multline}
For later use, we shall record several other properties the of $Q_\nu$ operators. First, as in \eqref{commute} and \eqref{commutee0}, we can use the arguments in \cite{SBLog} to see that 
\begin{equation}\label{commutetori} 
\bigl\|  \tilde\sigma_\la Q_\nu -Q_\nu 
 \tilde\sigma_\la 
\bigr\|_{L^2(M)\to L^{q_c}(M)} \lesssim
\la^{\frac1{q_c}-\frac12}T^{2},
\end{equation}
Also \eqref{add} and \eqref{alor} hold for the $Q_\nu$ operators defined in \eqref{qnu}, as long as $\delta>0$ in \eqref{i.2.2} is small enough. That is 
\begin{equation}\label{addtn}
\tilde \sigma_\la -\sum_\nu \tilde \sigma_\la Q_\nu
=R_\la \quad
\text{where } \, \|R_\la\|_{L^p\to L^p}=O(\la^{-N})
\quad \forall \, N \, \, \text{if } \, \,
1\le p \le \infty,
\end{equation}
and
\begin{equation}\label{alortn}
\sum_\nu \|Q_\nu h\|^2_{L^2(M)}\le C\|h\|^2_{L^2(M)}.
\end{equation}
Furthermore, as
in \eqref{7}, we have the following
``global'' {\em ``$L^{q_c}$-microlocal Kakeya-Nikodym estimates":}
\begin{equation}\label{globalqnu}
\sup_\nu \|Q_\nu \rho_\la\|_{L^2(\T^n)\to L^{q_c}(\T^n)}\lesssim_{\delta_0} \la^{\frac1{q_c}} \cdot T^{-\frac1{q_c}},\,\,\,\forall\,1\le T\le \la^{1-\delta_0}.
\end{equation}
Here $\delta_0$ is a fixed constant that can be arbitrarily small,  the condition $ T\le \la^{1-\delta_0}$ ensures that $\la^{-1}T\le \la^{-\delta_0}$, which implies that the $Q_\nu$ operator is rapidly decreasing outside a $\la^{-\delta_0}$ neighborhood of the diagonal. The estimate
\eqref{globalqnu} is analogous to \eqref{7} where $T=c_0\log\la$ there, while here in the torus case,  we can take $T$ to be almost as large as $\la$. 

If  $\Psi=\rho^2$, then \eqref{globalqnu} is equivalent to the estimate
\begin{equation}\label{tnn.3}
\bigl\|\, Q_\nu \Psi(T(\la-P))Q^*_\nu\, \bigr\|_{L^{q_c'}(M)\to L^{q_c}(M)} \lesssim \la^{2/q_c} \, T^{-2/q_c},
\end{equation}
By \eqref{tnn.2}, this would be a consequence of 
\begin{equation}\label{tnn.5}
\bigl\| \, Q_\nu \Psi(T(\la-P))\, \bigr\|_{L^{q_c'}(M)\to L^{q_c}(M)}\lesssim \la^{2/q_c} \, T^{-2/q_c}.
\end{equation}
If we repeat the arguments in \eqref{n.5}--\eqref{n.10},
in order to prove \eqref{tnn.5}, it suffices to show that
\begin{equation}\label{tnn.6}
\|Q_\nu \tilde G_{\la,\mu}\|_{L^{q'_c}(M)\to L^{q_c}(M)}\le C\la^{2/q_c} \, T^{-1} \, \mu^{\frac2{n+2}},
\end{equation}
where for $\mu=2^j\ge1$
\begin{equation}\label{tnn.7}
\tilde G_{\la,\mu} =\frac1{\pi T}\int e^{i\la t} \cos(t\sqrt{-\Delta_g}) \, \beta(|t|/\mu) \, \Hat \Psi(t/T) \, dt,
\end{equation}
since, after adding up the dyadic estimates for $\mu=2^j\le T$, this gives us the right side of \eqref{tnn.5}. And as before, the other operators involving $L_\la$ and the difference $G_\la-\tilde G_\la$ satisfy better estimates.

As in \eqref{n.11}, by using \eqref{tnn.2}, it is not hard to see that 
\begin{equation}\label{tnn.8}
\|Q_\nu \tilde G_\la\|_{L^{2}(M)\to L^{2}(M)}=O(\mu/T).
\end{equation}
Thus, by interpolation, \eqref{tnn.6} would be a consequence of 
\begin{equation}\label{tnn.9}
\|Q_\nu \tilde G_\la\|_{L^{1}(M)\to L^{\infty}(M)}\lesssim T^{-1}\la^{\frac{n-1}2} \mu^{1-\frac{n-1}2}.
\end{equation}
As before, we prove \eqref{tnn.9} by lifting the calculations up to the universal cover, where in our current case the universal cover is just $\R^n$ with the standard metric. More explicitly,  we shall use the fact that 
\begin{equation}
    \bigl(\cos t\sqrt{-\Delta_{\T^n}}\bigr)(x,y)
=\sum_{j\in \mathbb{Z}^n}\bigl(\cos t\sqrt{-\Delta_{\R^n}}\bigr)( x-(y+ j_\ell)),
\end{equation}
if $j_\ell=(\ell_1j_1, \ell_2j_2,\dots,\ell_nj_n)$ for $j=(j_1, j_2,\dots,j_n)\in \mathbb{Z}^n$. This also follows from the classical Poisson summation formula. Here we are abusing notation a bit by viewing $x\in \T^n$ as a point in  $\R^n$, since the torus can always be identified with the cube $Q=[-\ell_1/2,\ell_1/2]\times\dots\times [-\ell_n/2,\ell_n/2] \in \R^n$, which is a fundamental domain for the covering.
As a result, 
if $K_{\la,\mu}(x,y)$ denotes the kernel of $\tilde G_{\la,\mu}$ then we have the formula
\begin{equation}\label{tnn.13}
K_{\la,\mu}(x,y)=\sum_{j\in \mathbb{Z}^n}\tilde K_{\la,\mu}( x, (y+j_\ell)),
\end{equation}
if 
\begin{equation}\label{tnn.14}
\begin{aligned}
\tilde K_{\la,\mu}( x, y)&=\frac1{\pi T} \int e^{i\la t} \, \bigl(\cos t\sqrt{-\Delta_{\R^n}}\bigr)( x, y)
\, \beta(|t|/\mu) \, \Hat \Psi(t/T) \, dt.  \\
&=\frac1{\pi T}\cdot(2\pi)^{-n} \int \int_{\R^n}e^{i\la t}e^{i\langle x-y,\xi\rangle} \, \bigl(\cos t|\xi|\bigr)
\, \beta(|t|/\mu) \, \Hat \Psi(t/T) \, dtd\xi. 
\end{aligned}
\end{equation}
Note that by finite speed of propagation, the above kernel vanishes
if $|x-y|\ge 2\mu$, since $\beta(|t|/\mu)=0$ for $|t|\ge 2\mu$.  Also, since this function also vanishes for $|t|\le \mu/2$, by integrating by parts, it is straightforward to check that $\tilde K_{\la,\mu}(x, y)=O(\la^{-N})$ for every $N$ if
$| x- y|\le \mu/4$. Thus, we may assume that $|x-y|\approx \mu$. In this case, by using a stationary phase argument, we have  
\begin{multline}\label{tnn.15}
    \tilde K_{\la,\mu}( x, y)=\sum_{\pm} T^{-1}\la^{\frac{n-1}2}\mu^{-\frac{n-1}{2}} e^{\pm i\la|x-y|} a_\pm(\la,|x-y|)  +O(\la^{-N}) \\ \text{if}\,\,|x-y|\approx \mu \ge 1,
\end{multline}
where for $j\ge 0$, $
    |\partial_r^j a_\pm(\la,r) |\le C_j r^{-j}, \,\,r\ge \la^{-1}.
$

Since for fixed $x,y$ the number of choices of $j_\ell$ such that $|x-(y+j_\ell)|\approx \mu$ is $O(\mu^n)$. By summing up all possible choices of $j_\ell$, \eqref{tnn.15} also implies that 
\begin{equation}\label{tnn.16}
   |K_{\la,\mu}(x,y)|\lesssim T^{-1}\la^{\frac{n-1}2}\mu^{\frac{n+1}{2}},
\end{equation}
which is equivalent to $\|\tilde G_{\la,\mu}\|_{L^1\rightarrow L^\infty}=O(T^{-1}\la^{\frac{n-1}2}\mu^{\frac{n+1}{2}})$.

Note that the above estimate is not sufficient for proving \eqref{tnn.9}.  To prove \eqref{tnn.9}, as in the negative curvature case, we need to use the fact that \eqref{tnn.9} includes the microlocal cutoff which means that the kernel of 
$Q_\nu \tilde G_{\la,\mu}$
\begin{equation}\label{tnn.17}
\begin{aligned}
\sum_{j\in \mathbb{Z}^n}&\frac1{\pi T} \int e^{i\la t} \, \bigl(Q_\nu\cos t\sqrt{-\Delta_{\R^n}}\bigr)( x, y+j_\ell)
\, \beta(|t|/\mu) \, \Hat \Psi(t/T) \, dt \\
=\sum_{j\in \mathbb{Z}^n}& K_{\la,\mu.\nu}( x, (y+j_\ell))
\end{aligned}
\end{equation}
only involves $O(\mu)$ non-trivial terms for $\mu\gg 1$. In the above $Q_\nu$ denotes the pullback of the operator $Q_\nu$ to  the fundamental domain $Q$, which has the same Schwartz kernel as in \eqref{qnu}.

More explicitly, for fixed $x,y$ and $\mu$, if we define
\begin{equation}\label{jellcondition}
 D_{main}= \Big\{ j\in\mathbb{Z}^n:  \Big|\pm\frac{x-(y+j_\ell)}{|x-(y+j_\ell)|}-\nu\Big|\le C \mu^{-1},\,\,\, |x-(y+j_\ell)|\approx \mu\Big\},
\end{equation}
and 
\begin{equation}\label{jellcondition1}
 D_{error}=\{ j\in\mathbb{Z}^n : \Big|\pm\frac{x-(y+j_\ell)}{|x-(y+j_\ell)|}-\nu\Big|\ge C \mu^{-1},\,\,\, |x-(y+j_\ell)|\approx \mu\},
\end{equation}
for some large enough constant $C$. Then the contributions from the terms  $j\in D_{error} $ in \eqref{tnn.17} is negligible.
To see this, note that by \eqref{qnu} and  \eqref{tnn.15}, modulo $O(\la^{-N})$ errors, the kernel of $ K_{\la,\mu.\nu}( x, y)$ is 
\begin{equation}\label{qnugla}
\begin{aligned}
 C_{T,\la,\mu}\iint e^{i\la \langle x-z,\xi\rangle\pm i\la|z-y|}a_\pm(\la,|z-y|) \psi(x)\tilde\psi(z) \, \beta_\nu\big(\xi/|\xi|\bigr) \, \tilde\beta(|\xi|)d\xi dz,
   \end{aligned}
\end{equation}
where $C_{T,\la,\mu}=(2\pi)^{-n}\la^{\frac{3n-1}2}T^{-1}\mu^{-\frac{n-1}{2}}$ and 
$\tilde\psi \in C^\infty_0(\Omega)$ equals
one in a neighborhood of the $x$-support of $\psi$. Here we introduced the extra the cut-off function $\tilde\psi(z)$, which is allowed since as mentioned before, the $Q_\nu$ operator is rapidly decreasing outside a $\la^{-\delta_0}$ neighborhood of the diagonal. And since for fixed $y$, as in \eqref{tnn.15}, the integral is taken over the region $|z-y|\approx u \gg 1$, we can also assume that $|x-y|\approx u$.

Note that fixed $x$ and $y$ with $|x-y|\approx \mu$, if 
\begin{equation}\label{xycondition}
    \Big|\pm\frac{x-y}{|x-y|}-\nu\Big|\ge C \mu^{-1},
\end{equation}
for some large enough constant $C$, then it is not hard to check that for all $z$ in the support of $\tilde \psi$ and $\xi$ in the support of $\beta_\nu$ 
\begin{equation}\label{zxicondition}
    \Big|\pm\frac{z-y}{|z-y|}-\xi\Big|\ge C(2\mu)^{-1}.
\end{equation}
Thus, integration by parts in $z$ in \eqref{qnugla} yields 
\begin{equation}\label{qnuglatri}
    | K_{\la,\mu.\nu}( x, y)|\lesssim_N (\la\mu^{-1})^{-N} \la^{\frac{3n-1}2}T^{-1}\mu^{-\frac{n-1}{2}},
\end{equation}
which implies that 
\begin{equation}\label{qnuglatri1}
  \sum_{j\in D_{error}}   | K_{\la,\mu.\nu}( x, y+j_\ell)|\lesssim_N (\la\mu^{-1})^{-N} \la^{\frac{3n-1}2}T^{-1}\mu^{\frac{n+1}{2}}.
\end{equation}
Since we are assuming $\mu\le T\le \la^{1-\delta_0}$, \eqref{qnuglatri1} is better than desired when $N$ is large enough. 

On the other hand, it is straightforward to check that $\# D_{main}=O(\mu)$, and for each fixed $j\in D_{main}$, 
\begin{equation}\label{tnn.18}
\begin{aligned}
| K_{\la,\mu.\nu}( x, y+j_\ell)|&=\Big|\frac1{\pi T} \int e^{i\la t} \, \bigl(Q_\nu\cos t\sqrt{-\Delta_{\R^n}}\bigr)( x, y+j_\ell)
\, \beta(|t|/\mu) \, \Hat \Psi(t/T) \, dt \Big|\\
&= \Big|\int Q_\nu (x,z) \tilde K_{\la,\mu}( z, y+j_\ell) dz\Big|\\
&\le \sup_x \int|Q_\nu(x,z)|\, dz \cdot \sup_{z, y}|\tilde K_{\la,\mu}( z, y)|\\
&=O(T^{-1}\la^{\frac{n-1}2}\mu^{-\frac{n-1}{2}}),
\end{aligned}
\end{equation}
where used \eqref{tnn.1} and \eqref{tnn.15} in the last line. This implies that 
\begin{equation}\label{qnuglatri2}
  \sum_{j\in D_{main}}   | K_{\la,\mu.\nu}( x, y+j_\ell)|\lesssim \la^{\frac{n-1}2}T^{-1}\mu^{1-\frac{n-1}{2}}.
\end{equation}
Thus the proof of \eqref{tnn.9} is complete.

Recall that the standard orthonormal basis for Laplacian on the rectangular torus $\T^n$ is 
$$ \{e^{2\pi i(\frac{j_1}{\ell_1}x_1+\dots+\frac{j_n}{\ell_n}x_n)}(\ell_1\cdot\dots\ell_n)^{\frac12},\,\,\,j\in\mathbb{Z}^n\}.
$$
For later use, let us also define
\begin{equation}\label{qnu1}
    \tilde{Q}_\nu f=\sum_{j\in\mathbb{Z}^n} \eta_\nu( j_{\ell'}) a_j e^{2\pi i \langle j_{\ell'},\cdot x\rangle},\,\,\,\text{if} \,\,\,  f=\sum_{j\in\mathbb{Z}^n}  a_j e^{2\pi i \langle  j_{\ell'},\cdot x\rangle},
\end{equation}
where  $ j_{\ell'}=(\frac{j_1}{\ell_1},\dots,\frac{j_n}{\ell_n})
$ for $j=(j_1,\dots,j_n)\in \mathbb{Z}^n$ and $\eta_\nu \in C_0^\infty(\R^n)$  equals one in a neighborhood of the support of $\beta_\nu$ and $$\supp (\eta_\nu)=\{\xi\in\R^n:\big|\frac{\xi}{|\xi|}-\nu\big|\le 4T^{-1}\},
$$

It is obvious that $\tilde{Q}_\nu$ commute with the operators $\sigma_\la$ and $\rho_\la$, and 
\begin{equation}\label{qnu2}\sum_{\nu}\|\tilde{Q}_\nu f\|_2^2\lesssim \|f\|_2^2,
\end{equation}
if $\nu$ is chosen from a $T^{-1}$ separated set on ${S}^{n-1}$. 

On the other hand, we also have 
\begin{equation}\label{qqnu}
\|Q_\nu f- Q_\nu\tilde{Q}_\nu f\|_{L^p(\T^n)}\le \la^{-N} \|f\|_2, \,\,\, \forall\, p\ge 2, \,\,\, 1\le T\le \la^{1-\delta_0}.
\end{equation}
To see this, since $Q_\nu$ is rapidly decreasing outside a $\la^{-\delta_0}$ neighborhood of the diagonal, it suffices to show \eqref{qqnu} with $Q_\nu$ replaced by 
\begin{equation}\label{qnu3}
    Q_\nu(x,y)=\frac{\la^n}{(2\pi)^n}\int e^{i\la \langle x-y, \xi\rangle} \psi(x)\tilde\psi(y) \, \beta_\nu\big(\xi/|\xi|\bigr) \, \tilde\beta(|\xi|)d\xi,
\end{equation}
where $\tilde\psi \in C^\infty_0(\Omega)$ equals
one in a neighborhood of the $x$-support of $\psi$. Now, if $f=\sum_{j\in\mathbb{Z}^n}  a_j e^{2\pi i \langle j_{\ell'},\cdot x\rangle},$ 
\begin{multline}
    Q_\nu f- Q_\nu\tilde{Q}_\nu f\\=\frac{\la^n}{(2\pi)^n}\sum_{j\in \mathbb Z^n}\int e^{i\la \langle x-y, \xi\rangle+2\pi i \langle j_{\ell'}, y\rangle}(1-\eta_\nu(j_\ell))\psi(x)\tilde\psi(y) \, \beta_\nu\big(\xi/|\xi|\bigr) \, \tilde\beta(|\xi|)d\xi dy,
\end{multline}
integrating by parts in $y$ and using Sobolev give us \eqref{qqnu}.

We have collect all the necessary properties we need for $Q_\nu$. Now
write 
$$(\tilde \rho_\la h)^{2}= \, 
\diag(h)+ \far(h) ,
$$
where $\far(h)$ is defined as in \eqref{ndsum}, and
\begin{multline}\label{diagetorus}
\diag(h)=\sum_{(\nu',\tilde \nu')\in
\Xi_J}
\bigl(\tilde \sigma_\la Q_\nu h\bigr) \cdot
\bigl(\tilde \sigma_\la Q_{\tilde \nu}h\bigr), 
\\
 \text{if } \, \, 
T^{-1}\in (2^{J-4}, 2^{J-3}], \, \, \text{and } \quad h=\rho_\la f.
\end{multline}
Here as before, $v=(v', v_n)$, and the  definition of $\Xi_J$ is the same as the one in \eqref{org},  which include diagonal pairs where
$\nu'=\tilde \nu'$, and all $(\nu',\tilde \nu')\in
\Xi_J$ satisfy $|\nu'-\tilde \nu'|\lesssim T^{-1}$.
Thus, for each fixed $\nu'$, we have
$\#\{\tilde \nu': \, (\nu', \tilde \nu')\in
\Xi_J\}=O(1)$.

For $\far(h)$, we have
the favorable bilinear estimates
\begin{equation}\label{fartorus}
\int |\far(h)|^{q/2} \, dx
 \le C
\la \, 
\bigl(\la T^{-1}\bigr)^{\frac{n-1}2(q-q_c)}
\|h\|_{L^2(M)}^q, \, \, 
\text{if } \, \, q\in (\tfrac{2(n+2)}n, q_c).
\end{equation}
This is analogous to \eqref{far}, where we replaced $\la^{7/8}$ in \eqref{far} by $\la T^{-1}$ here. And as is pointed out below \eqref{far}, \eqref{fartorus} essentially follows form 
(4.7) in \cite{SBLog} up to rapidly decaying terms.

If we combine \eqref{fartorus} with the fact that $|\tilde \rho_\la f(x)|\le C(T\la)^{\frac{n-1}4}$
 on $A_-$, we have
 \begin{equation}\label{iitorus}
     \begin{aligned}
 \int_{A_-}|\tilde \sigma_\la h|^{\frac{q_c-q}2}| \, \far(h)|^{q/2} &\lesssim      \|\tilde \sigma_\la h\|_{L^\infty(A_-)}^{q_c-q} \,  \bigl(\la T^{-1}\bigr)^{\frac{n-1}2(q-q_c)} \cdot \la \\
&\le  \, \big(\la T\big)^{\frac{n-1}4(q_c-q)} \cdot\bigl(\la T^{-1}\bigr)^{\frac{n-1}2(q-q_c)} \cdot \la
\\
&= \bigl(\la^{-\frac{n-1}{4}}T^{\frac{3(n-1)}{4}}\bigr)^{(q_c-q)}\cdot \la.
     \end{aligned}
 \end{equation}
Thus, if one repeat the arguments in \eqref{3.2}--\eqref{bilinear 2} for the $Q_\nu$ operator defined in this section, then for dimensions $n\ge3$, we have
\begin{equation}\label{4.10}
\|\tilde\sigma_\la h\|_{L^{q_c}(A_-)}
\le C\la^{\frac1{q_c}} \mu_n(T,\la) \|h\|_{L^2(\T^n)}
+C\Bigl(\sum_\nu \|\tilde\sigma_\la \qn h\|_{L^{q_c}(\T^n)}^{q_c}\Bigr)^{1/q_c}, \, \, 
h=\rho_\la f.
\end{equation}
And for $n=2$, where $q_c=6$, this is replaced by
\begin{equation}\label{4.10'}
\|\tilde\sigma_\la h\|_{L^{6}(A_-)}
\le C\la^{\frac1{6}} \mu_2(T,\la) \|h\|_{L^2(\T^2)}
+C\Bigl(\sum_\nu \|\tilde\sigma_\la \qn h\|_{L^{6}(\T^2)}^{3}\Bigr)^{1/3}, \quad  h=\rho_\la f.
\end{equation}
Here $\mu_n(T,\la)=(\la^{-\frac{n-1}{4}}T^{\frac{3(n-1)}{4}})^{1-\frac{q}{q_c}}$, for any $\tfrac{2(n+2)}{n}< q< q_c$. If we require $\mu_n(T,\la)\le T^{-\frac{1}{q_c}}$ and let $q\downarrow \tfrac{2(n+2)}{n}$, we get the condition $T\le \la^{\frac{1}{n+3}-\epsilon}$ for arbitrary small $\e$.

\begin{proof}[Proof of Proposition~\ref{prop4.3}]

With $h=\rho_\la f$ in \eqref{4.10}, we have for $n\ge 3$, 
\begin{equation}\label{4.17}
\|\tilde \rho_\la f||_{L^{q_c}(A_-)} \le C\la^{\frac1{q_c}} T^{-\frac{1}{q_c}}\|f\|_{L^2(\T^n)}
+C\Bigl(\sum_\nu \|\tilde\sigma_\la \qn \rho_\la f\|_{L^{q_c}(\T^n)}^{q_c}\Bigr)^{1/q_c},
\end{equation}
if $1\le T\le\la^{\frac{1}{n+3}-\epsilon}$.
Note that the universal (local) spectral projection bounds from \cite{sogge88} yield
$$\|\tilde\sigma_\la\|_{L^2\to L^{q_c}}=O(\la^{\frac1{q_c}}).$$
Using this and \eqref{qqnu}, we have 
\begin{equation}\label{4.18}
\begin{aligned}
\sum_\nu & \|\tilde\sigma_\la \qn \rho_\la f-\tilde\sigma_\la \qn \rho_\la \tilde{Q}_\nu f\|_{L^{q_c}(\T^n)}^{2} \\
&=\sum_\nu  \|\tilde\sigma_\la (\qn - \qn \tilde{Q}_\nu) \rho_\la  f\|_{L^{q_c}(\T^n)}^{2}  \\
&\lesssim \la^{\frac2{q_c}} \sum_\nu \|(\qn- \qn \tilde{Q}_\nu) \rho_\la f\|_2^2 \\
&\lesssim \la^{\frac2{q_c}-N}\|\rho_\la f\|_2^2
\lesssim \la^{\frac2{q_c}-N}\| f\|_2^2.
\end{aligned}
\end{equation}
Thus, it suffices to show that 
\begin{equation}\label{4.21}
    \Bigl(\sum_\nu \|\tilde\sigma_\la \qn \rho_\la  \tilde{Q}_\nu f\|_{L^{q_c}(\T^n)}^{q_c}\Bigr)^{1/q_c} \le C\la^{\frac1{q_c}} T^{-\frac{1}{q_c}} \|f\|_2.
\end{equation}
Note that by \eqref{commutetori} and the fact that $\|\rho_\la\|_{L^2\to L^2}=O(1)$, we have
\begin{equation}\label{star1'}
\|\tilde \sigma_\la Q_\nu \rho_\la \tilde{Q}_\nu f\|_{L^{q_c}(M)}\lesssim
\|Q_\nu\tilde \sigma_\la \rho_\la \tilde{Q}_\nu f\|_{L^{q_c}(M)}+\la^{\frac1{q_c}-\frac12}T^{2}\|\tilde{Q}_\nu f\|_{L^2(M)}.
\end{equation}
As before, write $Q_\nu \tilde \sigma_\la =Q_\nu B\sigma_\la =
BQ_\nu\sigma_\la +[Q_\nu,B]\sigma_\la$, where in the current case the commutator $[Q_\nu,B]$ satisfies $\|[Q_\nu,B]\|_{L^{q_c}\to L^{q_c}}=O(\la^{-1}T^2)$.  So, by \eqref{blah} and \eqref{i.b3}
\begin{align}\label{star2'}
\|Q_\nu \tilde \sigma_\la \rho_\la \tilde{Q}_\nu f\|_{L^{q_c}(M)}
&\lesssim \|BQ_\nu \sigma_\la \rho_\la \tilde{Q}_\nu f\|_{L^{q_c}(M)}
+\la^{\frac1{q_c}-1}T^{2}\|\tilde{Q}_\nu f\|_{L^2(M)}
\\
&\lesssim \|Q_\nu \sigma_\la \rho_\la \tilde{Q}_\nu f\|_{L^{q_c}(M)}
+\la^{\frac1{q_c}-1}T^{2}\|\tilde{Q}_\nu f\|_{L^2(M)}.
\notag
\end{align}

If we combine \eqref{star1'} and \eqref{star2'}, by 
 \eqref{tilde}, \eqref{tnn.2} along with the Kakeya-Nikodym bounds \eqref{globalqnu}, we have
\begin{equation}\label{4.22}
    \begin{aligned}
    \|&\tilde\sigma_\la \qn \rho_\la  \tilde{Q}_\nu f\|_{L^{q_c}(\T^n)} \\
&\le \|\qn \sigma_\la \rho_\la  \tilde{Q}_\nu f\|_{L^{q_c}(\T^n)}+C
\la^{\frac1{q_c}-\frac12}T^{2}\| \tilde{Q}_\nu f\|_2
\\
&\lesssim \|\qn\rho_\la  \tilde{Q}_\nu f\|_{L^{q_c}(\T^n)} + \|Q_\nu\circ (I-\sigma_\la)\rho_\la  \tilde{Q}_\nu f\|_{L^{q_c}(\T^n)}
+C
\la^{\frac1{q_c}-\frac12}T^{2}\| \tilde{Q}_\nu f\|_2
\\
&\lesssim \|Q_\nu \rho_\la  \tilde{Q}_\nu f\|_{L^{q_c}(\T^n)} +\la^{\frac1{q_c}}T^{-1}\|  \tilde{Q}_\nu f\|_2+\la^{\frac1{q_c}} \,T^{-\frac{1}{q_c}} \, \| \tilde{Q}_\nu f\|_2
\\
&\lesssim \la^{\frac1{q_c}} \,T^{-\frac{1}{q_c}} \, \| \tilde{Q}_\nu f\|_2,
    \end{aligned}
\end{equation}
Here 
 we used the fact that $\la^{\frac1{q_c}-\frac12}T^{2}\le \la^{\frac1{q_c}}T^{-\frac{1}{q_c}}$ for $1\le T\le\la^{\frac{1}{n+3}-\epsilon}$. 

If we combine \eqref{qnu2} and \eqref{4.22}, we have 
\begin{equation}\label{4.23}
    \Bigl(\sum_\nu \|\tilde\sigma_\la \qn \rho_\la  \tilde{Q}_\nu f\|_{L^{q_c}(\T^n)}^{q_c}\Bigr)^{1/q_c} \le C\la^{\frac1{q_c}} T^{-\frac{1}{q_c}}  (\sum_\nu\| \tilde{Q}_\nu f\|^2_2)^{\frac12}\le C\la^{\frac1{q_c}} T^{-\frac{1}{q_c}} \|f\|_2,
\end{equation}
which implies \eqref{4.21}. The proof for the case $n=2$ is similar. If we combine Proposition~\ref{prop4.1} and Proposition~\ref{prop4.3}, we get \eqref{4.2}.
\end{proof}

\noindent {\bf Remark:}  In the next section, we shall show that  the results in Theorem~\ref{thm3} are optimal in terms of how the 
bounds depend on $T$.  
It would be interesting to see whether the optimal results in Theorem~\ref{thm3} are valid for
$T\lesssim \la^{1/2-}$.

\newsection{Geodesic concentration of quasimodes and lower bounds of $L^q(M)$-norms}\label{g}

In this section we shall show that the bounds in Theorem~\ref{thm3} are sharp in terms of their dependence on $T$ even though we expect the bounds to hold for a larger range of $T$, at least $T\in [1,\la^{1/2-}]$.  Bourgain and Demeter~\cite{BourgainDemeterDecouple} and Hickman~\cite{Hickman} (see also Germain and Myerson~\cite{Germain}) obtained near optimal bounds when $T$ is larger than a fixed power of $\la$; however, since their arguments used decoupling they involved 
$\la^\e$-losses for arbitrary $\e>0$.

We shall consider here 
relationships between the problem of obtaining lower bounds for the $L^q(M)$ norms of quasimodes $\Psi_\la$ of frequency $\la\gg 1$ on 
$M$  and the problem of establishing 
their
nontrivial $L^2$-mass concentration near a fixed
closed geodesic $\gamma\in M$.  Here, we are assuming that the spectrum associated with $P=\sqrt{-\Delta_g}$ of
the quasimodes $\Psi_\la$ belongs to intervals of the form $[\la-\e(\la),\la+\e(\la)]$, where, as $\la$ goes to infinity through a sequence
 $\la_{j} \to +\infty$ we have
$\e(\la)\searrow 0$.  We shall also assume that these quasimodes are $L^2$-normalized, i.e.,
\begin{equation}\label{qmnorm}
\|\Psi_\la\|_{L^2(M)}=1.
\end{equation}

Additionally, to get lower bounds for $L^q$-norms with $q>2$, if 
\begin{equation}\label{tube}{\mathcal T}_{\la,N}={\mathcal T}_{N(\la)\cdot \la^{-1/2}}(\gamma)=\{y\in M: \, d_g(y,\gamma)\le N(\la)\cdot \la^{-1/2}\},
\, \, \, N=N(\la),
\end{equation}
is a tube of width
$N(\la)\cdot \la^{-1/2}$ about our closed geodesic $\gamma$, we shall assume that we have the following uniform lower bounds for the
$L^2({\mathcal T}_{\la,N})$ mass of our quasimodes
\begin{equation}\label{tubemass}
\|\Psi_\la\|_{L^2({\mathcal T}_{\la,N})}\ge \delta_0>0.
\end{equation}
Typically, we shall have $N(\la)\to +\infty$ as $\la\to +\infty$, and \eqref{tubemass} with $N(\la)\equiv1$ corresponds to the concentration of
Gaussian beam quasimodes, which  is expected to be the tightest possible concentration about geodesics.
It is also not natural to consider tubes which are wider than the injectivity radius of 
$M$.   So, by \eqref{tube}, it is natural to assume, as we shall,
 that
\begin{equation}\label{tubewidth}
1\le N(\la) =O( \la^{1/2-}).
\end{equation}

Note that if we assume \eqref{tubemass}, then by H\"older's inequality we have for $q>2$
$$\delta_0\le \|\Psi_\la\|_{L^2({\mathcal T}_{\la,N})}\le |{\mathcal T}_{\la,N}|^{\frac12-\frac1q}\, \|\Psi_\la\|_{L^q(M)}.$$
Since the Riemannian volume of the tubes satisfies $|{\mathcal T}_{\la,N}|\approx \bigl(N(\la) \cdot \la^{-1/2}\bigr)^{(n-1)}$,  this gives the lower bounds
\begin{equation}\label{lower}
\bigl(N(\la)\bigr)^{-(n-1)(\frac12-\frac1q)} \, \la^{\frac{n-1}2 (\frac12-\frac1q)} \lesssim  \|\Psi_\la\|_{L^q(M)}
\end{equation}
for the lower bounds of our quasimodes satisfying \eqref{tubemass}.  Tighter concentration of $L^2$-mass near the geodesic (i.e. smaller $N(\la)$) implies stronger lower bounds.

Note that for the standard round sphere, $S^n$, the $L^2$-normalized highest weight spherical harmonics $Q_\la$ are {\em eigenfunctions}
satisfying \eqref{tubemass} with $N(\la)\equiv 1$, and the lower bounds in \eqref{lower} with $\Psi_\la=Q_\la$ are sharp since
$\|Q_\la\|_{L^q(S^n)}\approx \la^{\frac{n-1}2(\frac12-\frac1q)}$ if $q>2$, and this fact shows that the universal bounds in \cite{sogge88}
are best possible for $2<q\le q_c$.  (See \cite{sogge86}.) 

Let us also see here that we can also obtain  natural lower bounds for $L^q$ estimates for this range of $q$ if $M$ is a product manifold
with $S^1$ as one of its factors.  So let us assume for now that
\begin{equation}\label{product}
M=S^1\times X^{n-1}
\end{equation}
is a Cartesian product of $S^1$ with an $(n-1)$-dimensional compact manifold $X^{n-1}$ equipped with the product metric.
If $x_0\in X^{n-1}$ is fixed, let us consider the geodesic $\gamma=S^1 \times \{x_0\}$ in $M$.  Also, for $\la=1,2,3,\dots$ consider the
function on $S^1\times X^{n-1}$ defined by
\begin{equation}\label{qm}
\Psi_\la(\theta,x)=(\la^{1/2}/N)^{-(n-1)/2}e^{i\la \theta}  \,\beta\bigl(P_X/(\la^{1/2}/N)\bigr)(x_0,x), \quad N=N(\la),
\end{equation}
where $0\le \beta\in C^\infty_0((1/2,2))$ is a Littlewood-Paley bump function and $P_X=\sqrt{-\Delta_{X^{n-1}}}$ is the first order  operator on $X^{n-1}$ coming from its metric Laplacian.
Here $\beta(P_X/\mu)(x,y)$ is the kernel of the operator $\beta(P_X/\mu)$ on $X=X^{n-1}$ defined by the spectral theorem.

Let us now prove the following simple result on such product manifolds.

\begin{proposition}\label{productmfld}  Let $M=S^1\times X^{n-1}$ be as in \eqref{product} with $n\ge2$.
Then if $\e(\la)\searrow 0$ and \eqref{tubewidth} is valid with $N(\la)=(\e(\la))^{-1/2}$,  there is a geodesic $\gamma\in M$ so that
we have \eqref{tubemass}.
Consequently, we have
\begin{multline}\label{counterexample}
\|\chi_{[\la-\e(\la),\la+\e(\la)]}\|_{L^2(S^1\times X^{n-1})\to L^q(S^1\times X^{n-1})}  \ge c_0 \, \bigl(\e(\la)\cdot \la)^{\frac{n-1}2(\frac12-\frac1q)},
\\ q\in (2,q_c], \quad q_c=\tfrac{2(n+1)}{n-1}.
\end{multline}
\end{proposition}

\begin{remark}\label{toralremark}
Note that the universal bounds in \cite{sogge88} imply that the operators in \eqref{counterexample} map $L^2\to L^q$ with norm
$O\bigl(\la^{\frac{n-1}2(\frac12-\frac1q)}\bigr)$, and so \eqref{counterexample} says that for  manifolds as in \eqref{product} one can, at best, improve these bounds by
a factor of $\e(\la))^{\frac{n-1}2(\frac12-\frac1q)}$.  Consequently, 
since tori are product manifolds involving $S^1$ (perhaps scaled) as a factor,
the bounds in Theorem~\ref{thm3} are sharp for the
range $\e(\la)\in [\la^{-\frac1{n+3}+}, 1]$ there.  
\end{remark} 

As we pointed out before, Hickman~\cite{Hickman} showed that for tori 
if one includes additional arbitrary $\la^\e$ factors in the right side of \eqref{counterexample} then this
controls the bounds of the projection operators there.  Further results were obtained by Germain and Myerson~\cite{Germain}
as well as the lower bound \eqref{counterexample} for tori using a slightly different and more direct method.

\begin{proof}  We shall take $\gamma=S^1\times \{x_0\}$ as mentioned above.
We then notice that, if $\{\mu_k\}$ are the eigenvalues of $P_X$ and $\{e_k\}$ is the associated orthonormal basis, then 
\begin{align*}
\bigl\|\beta\bigl(P_X/(\la^{1/2}/N)\bigr)(x_0, \, \cdot \, )\|_{L^2(X^{n-1})}^2
&=\sum_k \beta^2(\mu_k/(\la^{1/2}/N))\, |e_k(x_0)|^2
\\
&\approx \# \{\mu_k: \, \mu_k\approx (\la^{1/2}/N)\} \approx (\la^{1/2}/N)^{n-1}. 
\end{align*}
Therefore, the functions in \eqref{qm} have $L^2(S^1\times X^{n-1})$ norms which are comparable to one.

Consequently, by \eqref{lower} with $N(\la)=(\e(\la))^{-1/2}$ we would have the lower-bound \eqref{counterexample} if we could show that
for large  $\la\in {\mathbb N}$
\begin{equation}\label{a}
\text{Spec }\Psi_\la \subset [\la -C_0\e(\la), \la + C_0\e(\la)],
\end{equation}
for some fixed constant, as well as
\begin{equation}\label{b}
\|\Psi_\la\|_{L^2(\{y\in S^1\times X^{n-1}: \, d_g(y,\gamma)\le
(\e(\la)\la)^{-1/2}\})} \ge \delta_0, \quad 
\gamma = S^1\times \{x_0\}, \, \, x_0\in X^{n-1},
\end{equation}
for some fixed $\delta_0>0$.  Here $\text{Spec }\Psi_\la$ of course denotes the $P$-spectrum of $\Psi_\la$.

To prove \eqref{a} we note that $\Psi_\la(\theta ,x)$ is a linear combination of eigenfunctions on $S^1\times X^{n-1}$ of the form
$e^{i\la\theta}e_k(x)$ where $\la\in {\mathbb N}$ is fixed and $\mu_k\approx \la^{1/2}/N(\la)=\sqrt{\e(\la)\la}$. Since
$$\sqrt{-\bigl(\tfrac{\partial^2}{\partial \theta^2}+\Delta_X\bigr)} \bigl(e^{i\la\theta}e_k(x)\bigr)
=\sqrt{\la^2+\mu_k^2} \,  \bigl( e^{i\la\theta}e_k(x)\bigr)$$
and
$$ \sqrt{\la^2+\mu_k^2} -\la =O(\e(\la)) \quad \text{if } \, \, \mu_k\le C\sqrt{\e(\la)\la},$$
we conclude that \eqref{a} is valid.  

To prove \eqref{b} it suffices to see that if $B(x_0,r)$ is the geodesic ball of radius $r\ll 1$ about $x_0$ in 
$X^{n-1}$ then we have the lower bound
\begin{equation}\label{mass}
\bigl(\e(\la)\la\bigr)^{-\frac{n-1}4} \, 
\bigl\| \beta\bigl(P_X/(\e(\la)\la)^{1/2}\bigr)(x_0,\, \cdot \, )\|_{L^2(B(x_0, \, (\e(\la)\la)^{-1/2}))} \ge \delta_0>0.
\end{equation}
This is standard calculation.  Since $\e(\la)\la\ge \la^\sigma$ for some $\sigma>0$ because of our assumption 
\eqref{tubewidth} one can argue as in 
\cite [\S 4.3]{SFIOII}
to see that for large $\la$ we have
$$ \beta\bigl(P_X/(\e(\la)\la)^{1/2}\bigr)(x_0,y ) \approx (\sqrt{\e(\la)\la})^{n-1}
\quad \text{if } \, \, \, y\in B(x_0, \, c_0 (\e(\la)\la)^{-1/2})$$
for some fixed $c_0>0$, which yields \eqref{mass} and
completes the proof.
\end{proof}


\noindent{\bf Concentration problems for
manifolds of negative curvature}

As we pointed out before the analog of \eqref{counterexample} with $\e(\la)\equiv 1$ is valid
on $S^n$ since there are $L^2$-normalized eigenfunctions on $S^n$ for which \eqref{tubemass} is valid with
$N(\la)\equiv 1$, which means they have a fixed fraction
of their $L^2$-mass in $\la^{-1/2}$ tubes about a geodesic.  The proof of Proposition~\ref{productmfld} showed that 
on tori one can find $\e(\la)$-quasimodes with fixed
$L^2$-mass in $(\e(\la))^{-1/2}\la^{-1/2}$ tubes around
a closed geodesic whenever $\e(\la)\in [\la^{-1+},1]$.
So we have natural concentration results for quasimodes
near geodesics in manifolds of positive and zero curvature.

We wonder if there are any analogous results for manifolds of {\em negative} curvature.  Recently there
has be somewhat related work by Brooks~\cite{BrooksQM},
Eswarathasan and Nonnenmacher~\cite{ENQM} and others
who showed that there is strong scarring of logarithmic quasimodes on compact
quotients of ${\mathbb H}^n$.  These say that, on such
manifolds, if $\gamma$ is a closed geodesic then there
is a sequence of $L^2$-normalized quasimodes $\Psi_\la$
satisfying \eqref{a} (or a related $L^2$-quasimode
condition) so that the quantum probability measures
$|\Psi_\la|^2 \, dVol_g$ exhibit strong scarring
along $\gamma$.  By this we mean that weak limits of
these measures must include a positive multiple of
the Dirac mass on $\gamma$, and so the $\Psi$ are 
tightly concentrated on this geodesic.

It would be interesting if one could find more
quantitative versions of these results that might
lead to nontrivial lower bounds for the $L^2\to L^{q_c}$
operator norms of the operators in \eqref{counterexample} with $\e(\la)=(\log\la)^{-1}$ (or perhaps larger).  
So it would be interesting to see if for certain
$N(\la)\nearrow$ we could find logarithmic quasimodes
$\Psi_\la$ (so $\e(\la)\approx (\log\la)^{-1}$)
on hyperbolic quotients for which we
have \eqref{tube}.   It is routine to see that on
{\em any} manifold one cannot have for $\e(\la)\searrow 0$ 
\begin{equation}\label{c}
\bigl\| \chi_{[\la-\e(\la),\la+\e(\la)]}\bigr\|_{L^2\to L^{q_c}} =o(\sqrt{\e(\la)}\la^{\frac1{q_c}}),
\end{equation}
since these bounds would imply that the operators map
$L^2\to L^\infty$ with norm $o(\sqrt{\e(\la)}\la^{\frac{n-1}2})$, and this would violate the 
Weyl counting formula.

Even though Hassell and Tacy~\cite{HassellTacy} showed
that there is $(\log\la)^{-1/2}$ improvements over
the universal bounds in \cite{sogge88} on the 
$L^2\to L^q$ norms of these operators for 
$q>q_c$ when $M$ has negative curvature and
$\e(\la)=(\log\la)^{-1}$, given
Proposition~\ref{productmfld} one might expect
such an improvement to not hold when $q$ is the critical
exponent $q=q_c$.  This would follow from showing
that there is a closed geodesic $\gamma$ and 
$L^2$-normalized logarithmic quasimodes as described
above so that one has the non-trivial
lower bounds \eqref{tubemass} with
$N(\la)=o((\log\la)^{q_c/4})=o((\log\la)^{\frac{n+1}{2(n-1)}})$.  So, for instance, when $n=2$, if one
could find $\Psi_\la$ as above with fixed
lower bounds of $L^2$-mass in 
$o((\log\la)^{3/2}\la^{-\frac12})$ neighborhoods
of a closed geodesic $\gamma$, then the 
$\sqrt{\e(\la)}=(\log\la)^{-1/2}$ improvement 
of Hassell and Tacy~\cite{HassellTacy} for $q>q_c$ 
could not hold for $q=q_c$.   

Also, it would be interesting to know whether our estimates in \eqref{116} for log-quasimodes on surfaces of negative curvature
are sharp.  If one could construct a sequence of such quasimodes for which $\|\Psi_\la\|_{L^2({\mathcal T})}\ge \delta_0$ for tubes ${\mathcal T}
={\mathcal T}_\la(\gamma)$   of width
$\log\la \cdot\la^{-1/2} $ about a periodic geodesic $\gamma$, then \eqref{116} would be saturated.  Note that these are much wider tubes
than the ones we used for $S^1\times X^{n-1}$, which is to be expected due to the much faster divergence of 
geodesics on manifolds of negative curvature.

As was kindly pointed out to us by Eswarathasan, due to
the role of the Ehrenfest time in the constructions,
it does not seem likely that the techniques in
\cite{BrooksQM} or \cite{ENQM} could be used to prove
this.  It would also be interesting to see whether
such a result might hold for log-log quasimodes, in which case one would replace the above $\log\la$ terms 
with $\log\log\la$.  So $\e(\la)=(\log\log\la)^{-1}$
and $N(\la)=o((\log\log\la)^{\frac{n+1}{2(n-1)}})$.
If we then had \eqref{tubemass} this would imply that the 
$\sqrt{\e(\la)}$ improvements for the 
$L^2\to L^q$ norms of the operators in \eqref{counterexample} of Hassell and Tacy for $q>q_c$
could not hold for $q=q_c$ with $\e(\la)=(\log\log\la)^{-1}$.

\bibliography{refs}
\bibliographystyle{abbrv}

%

\end{document}